\documentclass[12pt]{amsart}
\usepackage{amssymb}
\usepackage{amscd}
\usepackage{url}
\usepackage{hyperref}
\usepackage{color}

\numberwithin{equation}{section}

\theoremstyle{plain}

\newtheorem{theorem}[subsection]{Theorem}
\newtheorem{proposition}[subsection]{Proposition}
\newtheorem{lemma}[subsection]{Lemma}

\theoremstyle{definition}

\newtheorem{definition}[subsection]{Definition}

\newcommand{\supp}{\textrm{supp }}

\newcommand{\Z}{\mathbb Z}
\newcommand{\C}{\mathbb C}

\newcommand{\N}{\mathbb N}
\newcommand{\R}{\mathbb{R}}

\begin{document}

\title[Effective equidistribution]{Effective equidistribution of translates of large submanifolds in semisimple homogeneous spaces}

\author{Adri\'an Ubis}
\address{Departamento de Matem\'aticas \\ Universidad Aut\'onoma de Madrid \\ Madrid 28049 \\ Spain
}
\email{adrian.ubis@uam.es}

\thanks{}

\subjclass{}

\begin{abstract}
Let $G=SL_2(\R)^d$ and $\Gamma=\Gamma_0^d$ with $\Gamma_0$ a lattice in $SL_2(\R)$. Let $S$ be any  ``curved''  submanifold of small codimension of a maximal horospherical subgroup of $G$ relative to an $\R$-diagonalizable element $a$ in the diagonal of $G$. Then for $S$ compact our result can be described by saying that $a^n \text{vol}_S$  converges in an effective way to the volume measure of $G/\Gamma$ when $n\to \infty$, with $\text{vol}_S$ the volume measure on $S$.

\end{abstract}

\maketitle

\section{Introduction and results}

Let $G$ be a connected Lie group without compact factors and $\Gamma$ a lattice in $G$.  For any $a$ Ad-semisimple element of $G$, we can consider (see \cite{kleinbock_shah_starkov}) the \emph{expanding horospherical subgroup} relative to $a$
\[
 U^+ = \{ g\in G: \lim_{n\to \infty} a^{-n} g a^{n} = e \}.
\]
An element $u\in G$ is in $U^+$ whenever $d_{G}(a^{-n}u,a^{-n})\to 0$ as $n\to\infty$, where $d_G$ is  a fixed right $G$-invariant distance on $G$.  This says that the action $u\mapsto a^{-1}u$ contracts regions of $U^+$, so the opposite action $u\mapsto au$ expands regions of $U^+$.

Then, one would expect $a^nVx_0$ to be quite large inside $G/\Gamma$ for any $x_0\in G/\Gamma$ and $V$ open set in $U^+$. In fact, if $U^+$ is maximal the following much stronger equidistribution result is known \cite{veech,shah_horospheres} (see \cite[Theorem 3.7.8]{kleinbock_shah_starkov}):  for any probability measure $\lambda$ on $U^+$ which is absolutely continuous with respect to a Haar measure on $U^+$ we have $a^n\lambda^*\to \mu_G$
where $\lambda^*$ is the image of $\lambda$ onto $G/\Gamma$ under the map $g\mapsto gx_0$ for a fixed $x_0\in G/\Gamma$, and $\mu_G$ is the probability Haar measure on $G/\Gamma$, namely
\begin{equation}\label{horosphere_equidistribution}
 \lim_{n\to \infty} \int_{U^+} f(a^n u x_0) \, d\lambda(u)=\int_{G/\Gamma} f(r) \, d\mu_G(r)
\end{equation}
for any $f\in C_c(G/\Gamma)$.
In particular this implies $\overline{\cup_{n} a^n V x_0 } =G/\Gamma$. Actually the result in \cite{shah_horospheres} is much more general than (\ref{horosphere_equidistribution}).

In \cite{gorodnik,shah_icm} N. Shah raised the question of trying to generalize  this result to some singular measures $\lambda$ on $U^+$, in particular to find conditions on a submanifold $S$ of $U^+$  which make the  probability volume measure $\lambda=\lambda_S$ supported on $S$ satisfy (\ref{horosphere_equidistribution}). 

In \cite{shah_curves} this question was solved for the case $G=SO(d,1)$, where it was shown that (\ref{horosphere_equidistribution}) holds whenever $\gamma:[0,1]\to U^+$ is a real-analytic curve  with $\gamma([0,1])x_0$  not contained in any proper subsphere of $G/\Gamma$. This was extended to $C^m$ curves in \cite{shah_curves_cm}.

In this work we intend to study the case $G=SL_2(\R)^d$, $\Gamma=\Gamma_0^d$ with $\Gamma_0$ a lattice in $SL_2(\R)$, $a$ an element in the diagonal of $G$ and $S$ a  submanifold of $U^+$ of small codimension  not contained in an affine subspace of $U^+$.  Our methods will be Fourier-analytic and will give effective rates of decay; on the other hand, in order to prove (\ref{horosphere_equidistribution}) we will need to impose a curvature condition on $S$. I do not know whether the ideas from \cite{shah_curves} could be applied to this case. 

% {\color{black}
% 
%  In order to deal with it we will take advantage of the exponential mixing of the flow $g\Gamma\mapsto ag\Gamma$ with effective rates, which is well known and can be proven with and idea going back to \cite{margulis}. This can be used, as shown in \cite{venkatesh_sparse}, to prove effective uncorrelation with characters, which would be enough to attack the case of affine subspaces of $U^+$ of small codimension \cite[Remark 3.1]{venkatesh_sparse_arxiv}.
% 
% }

In $G=SL_2(\R)^d$ every semisimple element $a$ in the diagonal of $G$ that generates a maximal horospherical subgroup is conjugate to
\begin{equation}\label{translation_definition}
 a_y=(a(y),a(y),\ldots, a(y))    \qquad    a(y)=
 \begin{pmatrix}
  \sqrt y & 0 \\
  0 & 1/\sqrt y\\
 \end{pmatrix}
\end{equation}
with $0<y$, so it is enough to study the horospherical subgroup corresponding to that element with $0<y<1$, which is
\begin{equation}\label{horosphere_definition}
 U^+=\{u_t: t\in \R^d\}     \quad  u_t=(u(t_1),\ldots, u(t_d))  \quad  u(t)=
 \begin{pmatrix}
  1 & 0 \\
  t & 1 \\
 \end{pmatrix}.
\end{equation}
We then have that $U^+$ and $\R^d$ are isomorphic Lie groups, so we can think of $S$ as a submanifold of $\R^d$. Now we are going to impose on $S$ the following curvature condition, which is an strengthening of the fact that $S$ is not contained in any proper affine subspace of $\R^d$.

\begin{definition}[Totally curved submanifold]
Let $S$ be a submanifold of $\R^d$ of codimension $n\le d/2$. For any $p\in S$, we shall say that $S$ is \emph{totally curved at} $p$ if the second fundamental form (see \cite{kobayashi_nomizu})
\[
 \text{II}_p: T_p S\times T_p S \to (T_p S)^{\perp}
\]
satisfies $\text{II}_p(V\times T_p S)=(T_p S)^{\perp}$ for every $V$ subspace of $T_p S$ of dimension $n$. We shall say that $S$ is \emph{totally curved} if the set of points at which $S$ is curved is dense in $S$.
\end{definition}

Intuitively this condition is saying that the manifold is curved in every direction of $\R^d$.
In the case of $S$ being an hypersurface (namely $n=1$), $S$ is totally curved at $p$  precisely when it does not have zero curvature at that point. In general we shall show that for $S$ not to be totally curved its coordinates must satisfy a certain differential equation, so a generic submanifold of $\R^d$ will be totally curved. That equation is just $R(p)=0$, where $R$ is the complex resultant of the polynomials $s_j$, $0\le j\le n-1$ defined in Proposition \ref{totally_curved_char}.

An example of the exceptional submanifolds that we want to avoid with our curvature condition is $S=\gamma\times S_1$, with $\gamma$ a curve in $\R^2$ and $S_1$ a submanifold of $\R^{d-2}$; in order to prove (\ref{horosphere_equidistribution}) for $\lambda=\lambda_S$ we would need to prove it for $\lambda_{\gamma}$ in the case $d=2$, and then we would lose the small codimension condition for $S$. This indicates that our condition on $S$ is a natural one for our context.

As our results are quantitative, we need to control the smoothness of $f$ to measure the decay in (\ref{horosphere_equidistribution}); for that purpose we shall use Sobolev norms

\begin{definition}[Sobolev norms]
Any $X$ in the Lie algebra  $\mathfrak g$ of $G$ acts on $C^{\infty}(G/\Gamma)$ by {\color{black}$Xf(g\Gamma)=\frac{d}{dt} f (e^{tX}g\Gamma)|_{t=0}$.} Thus, by fixing a basis $\mathfrak B$ of $\mathfrak g$ we can define the $L^{\infty}$ \emph{Sobolev norms} as
\[
 \|f\|_{S^j}=\sum_{\text{deg}(\mathfrak{D}) \le j}  \| \mathfrak D f\|_{L^{\infty}(G/\Gamma, \mu_G)}
\]
for any $j\ge 0$, where $\mathfrak D$ runs over all the monomials in $\mathfrak B$ of degree at most $j$.
\end{definition}

With those definitions, we can already state  our main result.
\begin{theorem}[Main result]
\label{main}
Let $\Gamma_0$ be a lattice in $SL_2(\R)${\color{black}, with $0<r\le 1/2$ any lower bound for the smallest non-zero eigenvalue of the hyperbolic Laplacian on $\mathfrak H/\Gamma_0$. Then, for}  any $S$ real-analytic totally curved submanifold of {\color{black}$U^+\le SL_2(\R)^d$} of codimension {\color{black}$n<\frac{r^4}{700^2} d$}, with $U^+$ as in (\ref{horosphere_definition}), and any $\lambda_S$ probability measure  with $C_c^{\infty}$ Radon-Nikodym derivative w.r.t. the volume measure on $S$ we have
\[
 |\int_{U^+} f(a_yux_0) \, d\lambda_S(u) -\int_{G/\Gamma} f \, d\mu_{G} |\ll_{\lambda_S,x_0} y^{c} \|f\|_{S^d}
\]
for some $c>0$ and any $f\in C^{\infty}(G/\Gamma)$, $x_0\in G/\Gamma=SL_2(\R)^d/\Gamma_0^d$ and $a_y$ as in (\ref{translation_definition}) , where $c$ just depends on $d$ and $S$.
\end{theorem}

{\color{black}      
\begin{remark}
We think that our method could be extended to handle every lattice $\Gamma$ in $SL_2(\mathbb R)^d$, since the only thing that needs change is the part on mixing (section 2). For $\Gamma$ irreducible, we expect that the parameter $r$ would be replaced by a power of $1/p$, with $p=p(SL_2(\R)^d/\Gamma)$ the number defined on \cite{kelmer_sarnak}  that measures the spectral gap there; one would need to do as in \cite[Section 6]{venkatesh_margulis_einsiedler} to prove the necessary mixing rates. For the general lattice $\Gamma$, we know \cite{shimizu} that it must be commensurable to a completely reducible one $\prod_{j\le k} \Gamma_j$, with $\Gamma_j$ irreducible in $SL_2(\R)^{d_j}$. In that case we could take $p=\max_{j\le k} p(SL_2(\R)^{d_j}/\Gamma_j)$, by first reducing the problem to functions which are products of functions in $SL_2(\R)^{d_j}/\Gamma_j$, as we do in our case.
\end{remark}
}

In the Theorem, the $S^d$ norm can be {\color{black} replaced} by the $S^1$ norm at the expense of diminishing the constant $c$. As pointed out in \cite{shah_curves}, this result can be seen as a hyperbolic equivalent of the equidistribution in $\R^d/\Z^d$ of dilations of a real-analytic submanifold of $\R^d$ not contained in any proper affine subspace.

This result implies \eqref{horosphere_equidistribution} for $\lambda_S$ with a decay rate of $y^{cn}$. The idea of the proof is to see it as an equidistribution problem in $\R^d$, as in \cite{jones}, and to apply Fourier analysis there. Due to the conditions on $S$, it will be possible to show that $\widehat{\lambda}_S$ almost always behaves in a way similar to the Fourier transform of the sphere, and then one can take advantage of it by applying the Weyl-Van der Corput Method  together with the exponential mixing of one-parameter homogeneous flows of $G$ on $G/\Gamma$.

Actually, the approach described in the previous paragraph works just for functions $f$ for which there is a quick decay for the mixing. But for $G=SL_2(\R)^d$ this will always be the case except for functions coming from $SL_2(\R)^k$ with $k$ small, and for those the theorem can be directly proven.

In principle, our approach could be applied to any $G/\Gamma$ for which one has enough decay for the mixing, or if one can handle in another way the exceptional functions for which such a decay fails to exist.

Our Fourier-analytic arguments are local, in the sense that we use Stationary Phase to evaluate the Fourier Transform of the submanifold at each point. Perhaps it would be possible to extend our result to submanifolds with weaker curvature conditions by using global Fourier-analytic methods. On the other hand, curvature could be substituted by another kind of condition; for instance, in the case of $S$ being an affine subspace of large dimension, as suggested in \cite[Remark 3.1]{venkatesh_sparse_arxiv}, Fourier Analysis could be directly used to prove that our equidistribution result is true provided its  \emph{primitive dimension} is also large (see Definition \ref{Primitive manifold}).  Our methods cannot by themselves work for submanifolds of small dimension; the only advance on that difficult area seems to be \cite{venkatesh_margulis_einsiedler}.

Throughout the paper we will use the notation $e(t)=e^{2\pi i t}$ for $t\in\R$, $\widehat f(\xi)=\int_{\R^d} f(x) e(-\xi x) \, dx$ the Fourier Transform of $f$ in $\R^d$ with $\xi x$ the scalar product, and $f\ll g$ meaning $|f|\le C|g|$ for some positive constant $C>0$.

\section{Mixing and consequences}

We are going to deduce the exponential mixing for homogeneous one-parameter flows of $G$ on $G/\Gamma$ from the case $SL_2(\R)$ on $SL_2(\R)/\Gamma_0$. The action $h\Gamma\mapsto gh\Gamma$ of an element $g$ of $G$ on $G/\Gamma$ induces an action of $G$ on $C^{\infty}(G/\Gamma)$ described by $g\cdot f(h\Gamma)= f(gh\Gamma)$. On the other hand, we can consider the inner product in $L^2(G/\Gamma)$ with the Haar measure. With those definitions,  the following is a direct consequence of (9.6) in \cite{venkatesh_sparse} (for a more general result see \cite{kleinbock_shah_starkov})

\begin{lemma}[Exponential mixing for $SL_2(\R)$]
\label{mixing}
For any $g\in G$ and $f_1,f_2\in C^{\infty}(SL_2(\R)/\Gamma_0)$, with $\int f_2 \, d\mu_{SL_2(\R)}=0$ we have
\[
 \langle g\cdot f_1, f_2 \rangle \ll  \|g\|^{-r} \|f_1\|_{S^1} \|f_2\|_{S^1}
\]
where $\|g\|$ is the matrix norm of $g$ and $r$ is any number in $(0,1/2]$ {\color{black} that is a lower bound for  all nonzero eigenvalues of the hyperbolic Laplacian on $\mathcal H/\Gamma_0$.}
\end{lemma}

Now, we are going to use the previous result in order to prove results similar to Lemmas 3.1 and 9.4 in \cite{venkatesh_sparse}. Their proofs will also be like the ones there.

\begin{lemma}[Equidistribution of translates of horocycles]
\label{equidistribution_horocycles}
For any $x_0\in SL_2(\R)/\Gamma_0$, $f\in C^{\infty}(SL_2(\R)/\Gamma_0)$ with $\int f \, d\mu_{SL_2(\R)}=0$ and $\psi\in C_c^{\infty}(\R)$ we have
\[
 \int_{\R} f(a(y)u(t) x_0) \psi(t)  \, dt \ll y^{r/3}\|\psi\|_{S^1} \|f\|_{S^1},
\]
where the implicit constant depends on $x_0$ and the length of the smallest interval containing the support of $\psi$.
\end{lemma}
\begin{remark}
Here the Sobolev norms for $\psi$ and $f$ are the ones in $\R$ and $SL_2(\R)/\Gamma_0$ respectively.
\end{remark}

\begin{proof}
 The idea of the proof is that $a(y)$ expands the $u(t)$ direction and contracts the rest, so a ball flowed by $a(y)$ transforms essentially into a segment in $U^+$.

Let $\rho\in C_c^{\infty}(\R)$ with $\int_{\R} \rho=1$.   For $\delta\in (0,1)$ let us consider the measure
\[
 \nu_{\delta}(f)=\int_{\R^{3}} f(u(t_1)^t a(e^{t_2}) u(t_3) x_0)  \rho(t_1)\rho_{\delta}(t_2)\psi(t_3)  \, dt_1 dt_2 dt_3
\]
with $\rho_{\delta}(t)=\delta^{-1}\rho(t/\delta)$.  We can write
\[
 \nu_{\delta} (a(y)\cdot f)= \int_{\R^{3}} f(a(y)u(t_1)^t a(e^{t_2}) u(t_3) x_0)  \rho(t_1)\rho_{\delta}(t_2)\psi(t_3)  \, dt_1 dt_2 dt_3
\]
and since $a(y)u(t_1)^t=u(y t_1 )^t a(y)$ we get that
\[
 \nu_{\delta} (a(y)\cdot f)= \int_{\R^{3}} f(u(y t_1)^t a(e^{t_2}) a(y) u(t_3) x_0)  \rho(t_1)\rho_{\delta}(t_2)\psi(t_3)  \, dt_1 dt_2 dt_3.
\]
Now, by the mean value theorem (see \cite{venkatesh_sparse}, before Lemma 2.2) $|f(x_2)-f(x_1)|\ll \|f\|_{S^1} d(x_2,x_1)$, with $d$ a fixed right $SL_2(\R)$-invariant distance on $SL_2(\R)/\Gamma_0$, and then due to the identity $\int_{\R}\rho=1$ we have  
\[
\nu_{\delta}(a(y)\cdot f)=I+O(\delta+y)\|f\|_{S^1}\|\psi\|_{S^1}
\]
with $I$ the integral in the statement of the lemma.
On the other hand, any $g\in SL_2(\R)$ outside a set of measure zero can be uniquely written as $g=u(t_1)^t a(e^{t_2}) u(t_3)$ and then we can write $\nu_{\delta}(f)=\int_{SL_2(\R)} f(g) H_{\delta}(g)\, dg$ for some function $H_{\delta}\in C_c^{\infty}(SL_2(\R))$ depending on $x_0$, and covering $SL_2(\R)$ by translations of a fundamental domain of $SL_2(\R)/\Gamma_0$ we have
\[
 \nu_{\delta}(f)=\langle f, h_{\delta} \rangle,
\]
for some $h_{\delta}\in C_c^{\infty}(SL_2(\R)/\Gamma_0)$ with $\|h_{\delta}\|_{S^1}\ll \delta^{-2}\|\psi\|_{S^1}$. Here the implicit constant depends both on $x_0$ and on the length of the smallest interval containing the support of $\psi$. Thus by Lemma \ref{mixing} we have
\[
 \nu_{\delta}(a(y)\cdot f)=\langle a(y)\cdot f, h_{\delta}\rangle \ll y^r \|f\|_{S^1} \delta^{-2} \|\psi\|_{S^1}
\]
so by choosing $\delta=y^{r/3}$ we get the result.

\end{proof}

\begin{lemma}[Uncorrelation of translates of horocycles and characters]
\label{uncorrelation_horocycles_characters}
For $f\in C^{\infty}(SL_2(\R)/\Gamma_0)$ {\color{black} with $\int f \, d\mu_{SL_2(\R)}=0$ and}  $\psi\in C_c^{\infty}(\R)$ we have
\[
 \int_{\R} f(a(y)u(t) x_0) \psi(t) e(ct)  \, dt \ll  y^{\frac{r^2}{6+12r}} \|\psi\|_{S^1} \|f\|_{S^1}.
\]
where the implicit constant depends on $x_0$ and the length of the smallest interval containing the support of $\psi$ {\color{black} (in particular, the result is uniform in $c$).}
\end{lemma}
\begin{proof}
It is enough to prove the result for $\psi$ supported inside $(-1,1)$. Let $\varphi(t)=f(a(y)u(t)x_0)\psi(t)$. For any $v\in \R$ we can write the integral in the statement of the lemma as
\[
 I=\int_{\R} \varphi(t) e(ct) dt=\int_{\R} \varphi(t+v) e(ct) e(cv) dt
\]
so for $y<\delta<1$ and $\rho\in C_c^{\infty}(-1,1)$ with $\int \rho=1$ {\color{black} and $\|\rho\|_{L^{\infty}}\le 1$} we have
\[
 I= \int_{\R} \rho(s) \int_{\R} \varphi(t+\delta s) e(ct) e(c\delta s) dt   \, ds
\]
and then
\[
 |I|\le \int_{-2}^2  |\int_{\R} \varphi(t+\delta s) \rho(s) e(c\delta s) \, ds |  \, dt.
\]
By Cauchy's inequality
\[
 |I|^2\le 4 \int_{\R}  |\int_{\R} \varphi(t+\delta s) \rho(s) e(c\delta s) \, ds |^2 \, dt
\]
and by expanding the square and interchanging the integrals
\begin{align*}
 |I|^2  & \le 4 \int_{-1}^1 \int_{-1}^1   | \int_{\R}  \varphi(t+\delta s_1) \overline\varphi(t+\delta s_2)   \, dt   |  ds_1 ds_2, \\
 & \le 4  \int_{-2}^2   | \int_{\R}  \varphi(t+\delta s) \overline\varphi(t)   \, dt   |  ds.
\end{align*}
But
\[
  \int_{\R}  \varphi(t+w) \overline\varphi(t)   \, dt =\int_{\R} f(a(y) u(w) u(t) x_0) \overline f(a(y)u(t) x_0) \psi(t+w)\overline\psi(t) \,dt 
\]
and since $a(y)u(w)=u(w/y)a(y)$ we have
\[
  \int_{\R}  \varphi(t+w) \overline\varphi(t)   \, dt = \int_{\R} f_w (a(y)u(t) x_0) \psi_w(t) \, dt
\]
with $f_w(g\Gamma)=f(u(w/y)g\Gamma)\overline f(g\Gamma)$ and $\psi_w(t)=\psi(t+w)\psi(t)$. We can write
\[
 f_w(g\Gamma)= f_w^*(g\Gamma)+\int f_w(g\Gamma) d\mu_{SL_2(\R)}(g) = f_w^*(g\Gamma)+ \langle u(w/y)\cdot f, f\rangle
\]
with $\int f_w^* \, d\mu_{SL_2(\R)} =0$. Now, by applying Lemmas \ref{equidistribution_horocycles} and \ref{mixing}, and using that $\|f_w\|_{S^1}\ll \|u(\frac wy)\cdot f\|_{S^1}\|f\|_{S^1}\ll \frac wy\|f\|_{S^1}^2$ and $\|\psi_w\|_{S^1}\ll \|\psi\|_{S^1}^2$ \cite[Lemma 2.2]{venkatesh_sparse} we have
\[
 \int_{\R}  \varphi(t+w) \overline\varphi(t)   \, dt  \ll [ y^{r/3}(w/y)+ (y/w)^r] \|\psi\|_{S^1}^2 \|f\|_{S^1}^2.
\]
Using this bound for $|s|>(y/\delta)^{r/(1+r)}$  and the trivial bound otherwise we have
\[
 \int_{-2}^2   | \int_{\R}  \varphi(t+\delta s) \overline\varphi(t)   \, dt   |  ds\ll [(y/\delta)^{\frac{r}{1+r}} + y^{r/3} (\delta/y)] \|\psi\|_{S^1}^2 \|f\|_{S^1}^2
\]
so by picking $\delta$ such that both summands are equal we get the bound in the statement.
\end{proof}

In order to transfer the exponential mixing from $SL_2(\R)$ to $G$ in a simple way, we are going to deal just with functions in $G/\Gamma$ that can be written as products of functions in $SL_2(\R)/\Gamma_0$.

\begin{definition}[Factorizable function]
\label{factorizable}
We say that a function $f:G/\Gamma\to \C$ is \emph{factorizable} if it can be written as
\[
 f((g_1,\ldots, g_d)\Gamma)=f_1(g_1\Gamma)\ldots f_d(g_d\Gamma)
\]
with $f_j:SL_2(\R)/\Gamma_0\to \C$.
\end{definition}

\begin{lemma}[Equidistribution of translates of horospheres]
\label{equidistribution_horospheres}
Let  $f\in C_c^{\infty}(G/\Gamma)$ be a factorizable function with $k$ components having vanishing integral. Then, we have
\[
 \int_{\R^d} f(a_yu_tx_0) \psi(t) \, dt  \ll (y^{r/3})^k  \|\psi\|_{S^{4d}} \|f\|_{S^k}
\]
with the implicit constant depending on $x_0$ and the volume of the smallest ball containing the support of $\psi$.
\end{lemma}
\begin{proof}
 It is enough to prove it with the support of $\psi$ contained in the unit ball.  By choosing a fixed $\rho_*\in C_c^{\infty} ((-2,2))$ with $\rho_*=1$ in $(-1,1)$ and by setting $\rho(t)$ equal to $\rho_*(t_1)\ldots \rho_*(t_d)$
 we can write
 \begin{equation}\label{fourier_inv1}
  \psi(t)=\rho(t)\psi(t)=\int_{\R^d} \widehat\psi(\xi) \rho_{\xi}(t) \, d\xi   \qquad \rho_{\xi}(t)=\rho(t) e(\xi t)
 \end{equation}
By applying Lemma \ref{equidistribution_horocycles} on each factor with vanishing integral we have
\[
  \int_{\R^d} f(a_yu_tx_0) \rho_{\xi}(t) \, dt  \ll (y^{r/3})^k (1+|\xi|)^k \|f\|_{S^k}
\]
so the lemma follows from the bound $\widehat \psi(\xi)\ll_j \|\psi\|_{S^j} (1+|\xi|)^{-j}$.
\end{proof}

\begin{lemma}[Uncorrelation of translates of horospheres and characters]
\label{horospheres_un_characters}
Let $0<y<1$ and  $f\in C_c^{\infty}(G/\Gamma)$ be a factorizable function. Then for every $c=(c_1,\ldots,c_d)\in\R^d$ with $|c|\ll y^{-r^2/(24+48r)}$ we have
\[
  \int_{\R^d} f(a_yu_t x_0) \psi(t) e(c t)\, dt \ll  \tilde c^{-2d} \|\psi\|_{S^{4d}}\max_{{\color{black} 0\le l\le d}} [y^{\frac{r^2}{12+24r}l} \|f\|_{S^l}]
\]
with $\tilde c$ the geometric mean of $(1+|c_1|, \ldots, 1+|c_d|)$, with the implicit constant depending on $x_0$ and the measure of the smallest ball containing the support of $\psi$.
\end{lemma}
\begin{proof}
We begin as in the proof of Lemma \ref{equidistribution_horospheres}, by writing $\psi$ in terms of $\rho_{\xi}$, using (\ref{fourier_inv1}). Since $f$ is factorizable we have
\[
  \int_{\R^d} f(a_y u_t x_0) \rho_{\xi}(t) e(c  t)\, dt  =\prod_{j\le d} \int_{\R} f_j(a(y)u(t) x_{0,j}) \rho_{\xi, j}(t_j) e(c_jt_j) \, dt_j,
\]
so writing $f_j$ as a constant plus a function having  vanishing integral and applying Lemma \ref{uncorrelation_horocycles_characters} we have
\[
  \int_{\R^d} f(a_yu_t x_0) \rho_{\xi}(t) e(c t)\, dt \ll \prod_{j\le d} ( |\widehat \rho_{\xi,j}(c_j)| \|f_j\|_{S^0} + y^{\frac{r^2}{6+12r}}(1+ |\xi|) \|f_j\|_{S^1} ).
\]
Now we use the bound $|\widehat \rho_{\xi,j}(c_j)|\ll (1+|\xi|^2)/(1+ |c_j|^{2})$ and expand the product; in the resulting sum, each term with $l$ factors of the shape $y^{r^2/(6+12r)}(1+ |\xi|) \|f_j\|_{S^1}$ is bounded by
\[
 \tilde c^{-2d} (1+|\xi|)^{2d-l} [y^{\frac{r^2}{6+12r}} (1+|c|^2) ]^l \|f\|_{S^l},
\]
where we have used that the product of $d-l$ factors $(1+|c_j|)^{-2}$ is bounded by $(1+|c|)^{2l}/|\tilde c|^{2d}$. We finish by using the bound for $|c|$ and $\widehat\psi(\xi)\ll_j \|\psi\|_{S^j} (1+|\xi|)^{-j}$.
\end{proof}

\section{From $G/\Gamma$ to $\mathbb R^d$}

Our aim in this section is to transform the problem of equidistribution of translates of $S$ in $G/\Gamma$ to a related problem in $\R^d$.  For that we shall use the mixing results from the previous section; but first we need to show that for equidistribution it is enough to handle factorizable functions. This is proven in the next two lemmas.

\begin{lemma}[Reduction to bounded support]
\label{reduction_to_bounded_support}
Let $0<\delta<1$ and $\nu$ be a Borel probability measure on $G/\Gamma$. {\color{black} There exists a constant $C_{\Gamma}>0$ such that} if 
\begin{equation}\label{factorization_1}
| \int_{G/\Gamma} f\, d\nu -\int_{G/\Gamma} f \, d\mu_G| {\color{black} > } \delta \|f\|_{S^d}
\end{equation}
{\color{black}then the same  inequality is true replacing $\delta$ by $C_{\Gamma} \delta$ } and $f$ by either a factorizable function or a function supported on $B_R=B_{0,R}^d$ for some {\color{black}$R\ll_{\Gamma}\delta^{-1}$}, with $B_{0,R}$ the subset of elements $g\Gamma_0$ in $SL_2(\R)/\Gamma_0$ satisfying $\|g\Gamma_0\|^2\le R$, with $\|g\Gamma_0\|=\min_{\gamma_0\in\Gamma_0}\|g\gamma_0\|$, $\|\cdot\|$ the Frobenius matrix norm.
\end{lemma}
\begin{proof}
Pick $\psi\in C_c^{\infty}(SL_2(\R))$ non-negative with $\int_{SL_2(\R)} \psi(h)\, dh=1$, $dh$ a Haar measure on $SL_2(\R)$, and consider the convolution
\[
 \psi_{0,R}(g\Gamma_0)=\int_{SL_2(\R)} \psi(h) 1_{B_{0,R}}(h^{-1}g\Gamma_0)  \, dh.
\]
By the multiplicativity of $\|\cdot\|$ one can show that the supports of $\psi_{0,R}$ and $1-\psi_{0,R}$ are contained in $B_{0,c_1R}$ and $B_{0,c_2 R}^c$ respectively, with $c_1,c_2>0$ two absolute constants.
On the other hand, since 
\[
\psi_{0,R}(g\Gamma_0)=\int_{SL_2(\R)} \psi(gh^{-1})1_{B_{0,R}}(h\Gamma_0)\, dh
\]
we deduce that $\|\psi_{0,R}\|_{S^d}\ll 1$. Then, since  $\mu_{SL_2(\R)}(B_{0,R}^c)\ll R^{-1}$ \cite{iwaniec_book}, we can write
$
 f=f\psi_R+\sum_{i} f \epsilon_i\varrho_i
$
with $i\le 2^d-1$,  $\psi_{R}(g\Gamma)=\prod_{j\le d}\psi_{0,R}(g_j\Gamma_0)$, $\epsilon_i=\pm 1$ and $\varrho_i$ non-negative factorizable functions with $\|\varrho_i\|_{S^d}\ll 1$ and supported on a set of $\mu_G$ measure $O(R^{-1})$. By submultiplicativity of the $S^d$ norm and linearity we can {\color{black}replace} $f$ by either $f\psi_{R}$ or $f\varrho_i$ (for some $i$) in (\ref{factorization_1}). In the former case we are done because $f\psi_{R}$ is supported on $B_{R}$; in the latter, if $R\delta$ is larger than an absolute constant, taking into account the measure of the support of $\varrho_i$ and its non-negativity we can {\color{black}replace} $f\varrho_i$ by the factorizable function $\|f\|_{S^d} \varrho_i$ in (\ref{factorization_1}).
\end{proof}

\begin{lemma}[Reduction to factorizable functions]
\label{reduction_to_factorizable}
If we have (\ref{factorization_1}) then there exists a factorizable $f^*$ such that
\begin{equation}\label{factorization_2}
 | \int_{G/\Gamma} f^*\, d\nu - \int_{G/\Gamma} f^* \, d\mu_G | \gg \delta^{4d+1} \|f^*\|_{S^d}.
\end{equation}
\end{lemma}
\begin{proof}
By Lemma \ref{reduction_to_bounded_support} we can assume that the $f$ in (\ref{factorization_1}) is supported in $B_R$ for some $R\ll \delta^{-1}$. Now, for every $0<\beta<1$ consider the function $\eta_{0,\beta}\in C_c^{\infty}(SL_2(\R))$ defined as
$
 \eta_{0,\beta}(h)=c_{\beta}\phi(d_0(h,I)/\beta) 
$
with $\phi\in C_c^{\infty}(1,2)$, $d_0$ a fixed right $SL_2(\R)$ invariant Riemannian distance on $SL_2(\R)$ and $c_{\beta}$ a constant such that 
$\int_{SL_2(\R)} \eta_{0,\beta}(h) \, dh=1$. Thus $c_{\beta}\asymp \beta^{-3}$ and  $\eta_{0,\beta}$ satisfies $\|\eta_{0,\beta}\|_{S^j}\ll \beta^{-3-j}$. 
Let us define the factorizable function $\eta_{\beta}\in C_c^{\infty}(G)$ defined as $\eta_{\beta}(g)=\prod_{j\le d} \eta_{0,\beta}(g_j)$. We have $\|\eta_{\beta}\|_{S^j}\ll \beta^{-3d-j}$ and it is supported on a ball of radius $O(\beta)$ around the identity in $G$, with the metric $d$ induced from $d_0$. 
For any $g\in G$, let us consider the function $\eta_{\beta}^g(h)=\eta_{\beta}(gh)$. By invariance of the Haar measure we have $\int_G \eta_{\beta}^g(h)\, dh=1$. 

Consider the map $I: C_c(G)\to C_c(\Gamma\setminus G)$ defined as $I(w)(\Gamma h)=\sum_{\gamma\in \Gamma} w(\gamma h)$ and the left $G$ invariant measure $\mu^*$ on $\Gamma\setminus G$ defined by $\int_{\Gamma\setminus G} I(w) \, d\mu^* =\int_G w(h)\, dh$ (see \cite{raghunathan}). Applying this to $\eta_{\beta}^g$ we have
\begin{equation}\label{factorization_3}
 \int_{\Gamma\setminus G} I(\eta_{\beta}^g)(\Gamma h) \, d\mu^*(\Gamma h) =1.
\end{equation}
One can see that $I(\eta_{\beta}^g)(\Gamma h)$ is invariant under the change $g\mapsto g\gamma$, $\gamma\in \Gamma$, and then we can write $I(\eta_{\beta}^g)(\Gamma h)=\rho_{\beta,\Gamma h}(g\Gamma)$. Then,
using (\ref{factorization_3}) into (\ref{factorization_1}) and Fubini we have 
\begin{equation}\label{factorization_4}
 \int_{\Gamma\setminus G}|\int_{G/\Gamma} f\rho_{\beta,\Gamma h} \, d\nu -\int_{G/\Gamma} f\rho_{\beta,\Gamma h} \, d\mu_G|   \, d\mu^*(\Gamma h) \gg \delta \|f\|_{S^d}.
\end{equation}
Now, if $g\Gamma$ is in the support of $\rho_{\beta,\Gamma h}$ we have $d(g\gamma h, I)\ll \beta$ for some $\gamma\in\Gamma$  which by the mean value theorem (see \cite{venkatesh_sparse}, before Lemma 2.2) $|f(g\Gamma)-f(h^{-1}\Gamma)|\ll \|f\|_{S^1} \beta$. Hence, if $\beta/\delta$ is smaller than an absolute positive constant, from (\ref{factorization_3}), (\ref{factorization_4}) and H\"older inequality we deduce that 
\[
 |\int_{G/\Gamma} \rho_{\beta,\Gamma h} \, d\nu - \int_{G/\Gamma} \rho_{\beta,\Gamma h}  \, d\mu_G |\gg \delta
\]
for some $\Gamma h$ with $\|h^{-1}\Gamma\|^2\ll \delta^{-1}$, taking into account the support of $f$ and the fact that if $g\Gamma$ is in the support of $\rho_{\beta,\Gamma h}$ we have $d(g,h^{-1}\gamma^{-1})\ll 1$ for some $\gamma^{-1}\in \Gamma$. One can check that $\rho_{\beta,\Gamma h}$ is factorizable, so that (\ref{factorization_2}) will follow if we show that $\|\rho_{\beta,\Gamma h}\|_{S^j}\ll \| \eta_{\beta}\|_{S^j}\ll \beta^{-3d-j}$ with $\beta^{-1}\ll \delta^{-1}$. But this in turn follows from showing that in the sum defining $\rho_{\beta,\Gamma h}$ there is just one non-vanishing term. Suppose there were at least two non-vanishing terms. Then $d(g\gamma_1 h, I)\ll \beta$ and $d(g\gamma_2 h,I)\ll \beta$ for $\gamma_1\neq\gamma_2\in\Gamma$, $\|h\|^2\ll \delta^{-1}$. By the right invariance of the metric this implies $d(h^{-1}\gamma h,1)\ll \beta$ for $\gamma=\gamma_2^{-1}\gamma_1\neq 1$, hence $h^{-1}\gamma h=\exp(X)$ with $X\in\mathfrak g$, $\|X\|\ll \beta$. Therefore $\gamma=\exp(hXh^{-1})$ and $\|hXh^{-1}\|\ll \|h\| \|X\| \|h^{-1}\|\ll \beta/\delta $ since $\|h^{-1}\|\ll \|h\|$. Since $\Gamma$ is discrete, this gives a contradiction for $\beta/\delta$ small enough.
\end{proof}

We are going to write our problem in $G/\Gamma$ as a problem in $\R^d$, in the spirit of \cite{jones}, in order to use Fourier Analysis in $\R^d$ afterwards. By fixing a function $\rho\in C_c^{\infty}(\R^d)$ such that $\rho(x)=1$ whenever $u_x$ is in the support of $\lambda_S$, we have
\[
 \int_{U^+} f(a_yux_0) d\lambda_S (u)= \int_{\R^d} F(x) \, d\lambda_S(x)
\]
with $F(x)=\rho(x) f(a_yu_x x_0)$. We are going to denote this functional as
\[
\lambda_S(F)=\int_{\R^d} F(x) \, d\lambda_S(x);
\]
we will need to study its action on functions $F$ satisfying some special properties.

\begin{definition}[Singular function]
 We shall say that $f:\R^d\to \C$ is a $(T,\alpha)$-\emph{singular function} if $\|f\|_{S^0}\le 1$, $\|f\|_{S^1}\le T$, with
  \[
 | \int_{\R^d} f(x) \psi(x) \, dx |\le T^{-\alpha} \|\psi\|_{S^{4d}}
 \]
for any $\psi\in C^{\infty}(\R^d)$  and there exist $k\le d/2$ and $\tilde f:\R^k\to \C$ such that $f(x_1,x_2)=\tilde f(x_1) \tilde \rho(x_2)$ for every $(x_1,x_2)\in \R^k\times \R^{d-k}$, with $\tilde\rho\in C_c^{\infty}(\R^{d-k})$, $\|\tilde\rho\|_{S^1}\le 1\le \int_{\R^{d-k}}\tilde\rho$.
\end{definition}

\begin{definition}[Mixing function]
 We shall say that $f:\R^d\to \C$ is a $(T,\alpha)$-\emph{mixing function} if $\|f\|_{S^0}\le 1$, $\|f\|_{S^1}\le T$, $f$ supported in a translation of the unit ball, with
 \begin{equation}\label{horosphere_condition}
\int_{|v|<V} |\int_{\R^d} f(x+\frac vT) \overline f(x)\psi(x) \, dx | 
\, dv \le V^{d(1-\alpha)} \|\psi\|_{S^{4d}}
 \end{equation}
for every $V<T^{\alpha}$ and
\begin{equation}\label{character_condition}
 \int_{|h|<T^{\alpha}} \sup_{|v|<T^{\alpha}}|\int_{\R^d} f(x+\frac{v}{T}) \overline f(x) \psi(x)e(hx) \, dx | \, dh \le \|\psi\|_{S^{4d}}
\end{equation}
for any $\psi\in C^{\infty}(\R^d)$.
\end{definition}

Now, the main result of this section is
\begin{proposition}[From $G/\Gamma$ to $\R^d$]
\label{from_group_to_real}
 Let $y,\delta\in (0,1/2)$ and $f_*\in C^{\infty}(G/\Gamma)$ with
 \[
  |\int_{U^+} f_*(a_yux_0)\, d\lambda_S(u)-\int_{G/\Gamma} f_* \, d\mu_G|>\delta \|f_*\|_{S^d}.
 \]
 Then we have
 \[
  |\lambda_S(f)|\gg\delta^{4d+1}
 \]
for some function $f$  either $(y^{-1},\alpha)$-singular or $(y^{-1},\alpha)$-mixing for {\color{black}$\alpha=r^2/48$}.
\end{proposition}
\begin{proof}
By Lemma \ref{reduction_to_factorizable} we have
\begin{equation}\label{proof_1}
  |\int_{U^+} f_{**}(a_yux_0)\, d\lambda_S(u)-\int_{G/\Gamma} f_{**} \, d\mu_G|>\delta^{4d+1} \|f_{**}\|_{S^d}
\end{equation}
with $f_{**}$ a factorizable function. We can write $f_{**}-\int_{G/\Gamma} f_{**}\, d\mu_G$ as a sum of $O(1)$ factorizable functions with vanishing integral and each of the factors either constant or with integral zero, so from (\ref{proof_1}) we deduce that 
\[
 |\int_{U^+} \tilde f_*(a_yux_0)\, d\lambda_S(u)|\gg\delta^{4d+1} \|\tilde f_{*}\|_{S^d}
\]
with $\tilde f_*$ one those  $O(1)$ factorizable functions. By dividing $\tilde f_*$ by its norm we can assume that $\|\tilde f_*\|_{S^d}=1$, so defining 
\[
 f(x)=\rho(x)\tilde f_* (a_yu_xx_0)
\]
with $\rho\in C_c^{\infty}(\R^d)$ with $\rho=1$ whenever $u_x$ is in the support of $\lambda_S$, where $\rho(t)=\prod_{j\le d} \rho_0(t_j)$ with $\rho_0\in C_c^{\infty}(\R)$ and $\|\rho_0\|_{S^1}\le 1\le \int_{\R}\rho_0$,  we have $|\lambda_S(f)|\gg\delta^{4d+1}$ and then it only remains to show that $f$ satisfies the properties in the statement of the proposition. The bound on $\|f\|_{S^1}$ comes just from \cite[Lemma 2.2]{venkatesh_sparse}.

Let us assume first that at least $d/2$  components in the factorization of $\tilde f_*$ are constant.  We consider that the constant components are the last ones. Thus we get that $f$ is an $(y^{-1},\alpha)$-singular function for any $\alpha\le r/3$, since by Lemma \ref{equidistribution_horospheres} we have
\[
 \int_{\R^d} f(x) \psi(x) \, dx = \int_{G/\Gamma} \tilde f_*(a_yu_x x_0) \rho(x) \psi(x) \, dx \ll y^{r/3}  \|\psi\|_{S^{4d}}.
\]
Now, we have to consider the case in which $q\ge d/2$ components of $\tilde f_*$ have integral zero. By splitting $\rho$ into several functions we can assume that $\rho$ is supported in a translation of the unit ball. We have
\[
 \int_{\R^d} f(x+vy) \overline f(x) \psi(x) \, dx = \int_{\R^d} f_*^v (a_yu_x x_0) \rho(x+vy)\overline \rho(x)\psi(x) \, dx
\]
with $f_*^v (g\Gamma)= \tilde f_*(u_v g\Gamma) \overline{\tilde f_*}(g\Gamma)$, which is also factorizable. By writing each of its factors as a constant plus a function of vanishing integral and by expanding the product we see that $f_*^v$ is the sum of $O(1)$ special factorizable functions, so in order to prove (\ref{horosphere_condition}) for $f$ we can assume that $f_*^v$ is one of them. Then, from the $q$ factors coming from components of $\tilde f_*$ with zero integral, either $q/2$ of them are constant or $q/2$ of them have vanishing integral. In the first case\footnote{In a previous version of the paper I did not deal with this  case. I would like to thank Asaf Katz for pointing that error out to me.}, we can apply Lemma \ref{mixing} for each of those constants and condition (\ref{horosphere_condition}) follows for any $\alpha<rq/2d$, hence for any $\alpha<r/4$. In the second case by applying Lemma \ref{equidistribution_horospheres} and \cite[Lemma 2.2]{venkatesh_sparse} we have
\[
 \int_{\R^d} f(x+vy) \overline f(x) \psi(x) \, dx \ll (y^{r/3})^{q/2} \|f_*^v\|_{S^{q/2}}\|\psi\|_{S^{4d}} \ll (y^{r/3} |v|)^{q/2} \|\psi\|_{S^{4d}}
\]
and then condition (\ref{horosphere_condition}) follows for any $\alpha\le r/6$. Moreover
\[
 \int_{\R^d} f(x+vy)\overline f(x) \psi(x) e(hx) \, dx=\int_{\R^d} f_*^v (a_yu_x x_0) \rho(x+vy)\overline \rho(x)\psi(x) e(hx) \, dx
\]
so by Lemma \ref{horospheres_un_characters}  and \cite[Lemma 2.2]{venkatesh_sparse} we have {\color{black} $\|f_*^v\|_{S^l}\ll (1+|v|)^l$} and
\[
\int_{\R^d} f(x+vy)\overline f(x) \psi(x) e(hx) \, dx\ll   \tilde h^{-2d} \|\psi\|_{S^{4d}}
\]
for $|v|+|h|^2 \ll y^{-\frac{r^2}{12+24r}}$
and then (\ref{character_condition}) follows for every $\alpha\le r^2/(24+48r)$.

\end{proof}

We will dedicate the rest of the paper to prove the following result, which by Proposition \ref{from_group_to_real} implies Theorem \ref{main}{\color{black}, since $\frac{(r^2/48)^2}{200}>\frac{r^4}{700^2}$}:

\begin{theorem}[Main result for $\R^d$]
\label{main_real}
{\color{black}Let $0<\alpha<1$ and} let $f$ be either a $(T,\alpha)$-singular or a $(T,\alpha)$-mixing function. Let $S$ be a totally curved real-analytic submanifold of $\R^d$ of codimension $n<\alpha^2 d/200$, and $\lambda_S$ a probability measure with $C_c^{\infty}$ Radon-Nikodym derivative w.r.t. the volume measure on $S$. Then we have
\[
 |\lambda_S(f)|\ll T^{-c},
\]
with $c>0$ depending just on $S$ and $\alpha${\color{black}, and the implicit constant on the support of $\lambda_S$}  .

\end{theorem}

\section{Geometric properties of the submanifold}

We  begin by expressing that a manifold is totally curved in different ways.

\begin{proposition}[Totally curved characterizations]
\label{totally_curved_char}
Let $S=\varphi((-1,1)^m)$ with $\varphi:(-2,2)^m\to \R^d$, $\varphi(t)=(t,w(t))$, $w$ a real-analytic function. $S$ being totally curved at a point $p$ is equivalent to any of the following conditions (with $n=d-m$):
\begin{enumerate}
 \item For every subspace $V$ of $T_p S$ of dimension $n$ there exist $v_1,\ldots v_n$ in $V$ and tangent fields $X_1,\ldots, X_n$ such that
 \[
  \langle T_p S, D_{v_j} X_j(p): j\le n  \rangle = \R^d
 \]
 with $D_v$ the directional derivative in the direction $v$.
 
 \item For every subspace $V$ of $T_p S$ of dimension $n$ there exist functions $v_1,\ldots v_n$ smooth at $\varphi^{-1}(p)$ with $v_j(\varphi^{-1}(h))\in T_h S$ for $h$ near $p$ and $V=\langle v_j(\varphi^{-1}(p)): j\le n\rangle$ such that
 \[
  \langle T_p S, \partial_s v_j(\varphi^{-1}(p)): s\le m, j\le n\rangle =\R^d.
 \]

 \item For every nonzero $a\in\R^d$ orthogonal to $T_p S$ we have that the dimension of $\ker H_a$ is less than $n$, with $H_a:\R^m\to \R^m$ the Hessian of $a \varphi$ at $\varphi^{-1}(p)$.
 
 \item For every nonzero $z\in \R^n$ we have that the dimension of $\ker H_z$ is less than $n$, with $H_z$ the Hessian of $z w$ at $\varphi^{-1}(p)$.
 
 \item If $ \lambda^m+s_{m-1}(z) \lambda^{m-1}+\ldots+s_1(z)\lambda+s_0(z)$ is the characteristic polynomial of $H_z$, the Hessian of $zw$ at  $\varphi^{-1}(p)$, then the system of homogeneous polynomial equations $s_0(z)=s_1(z)=
\ldots =s_{n-1}(z)=0$ does not have a  solution $z\in\R^n\setminus \{0\}$.
\end{enumerate}
\end{proposition}
\begin{proof}
The equivalence with (i) and (ii) comes just from the definition of the second fundamental form by the equation
\[
 D_v X(p) = \nabla_v X(p) + \text{II}_p(v, X(p))
\]
with $\nabla_v X\in T_p S$ the covariant derivative on $S$. Moreover (iii), (iv) and (v) are clearly equivalent. Thus we just have to prove that (ii) and (iii) are equivalent.
 
Let us assume that (ii) is not satisfied. Then there exists $V$ such that
\[
\langle T_p S, \partial_s v_j(\varphi^{-1}(p)): s\le m, j\le n \rangle \neq \R^d.
\]
We can write $T_p S=\langle (\partial_s\varphi)(\varphi^{-1}(p)):s\le m\rangle$ and $v_j=\sum_{k\le m} r_j^k \partial_k\varphi$. Then we have
\[
 \partial_s v_j= \sum_{k\le m} \partial_s r_j^k \partial_s \varphi + u_j^s
\]
with $u_j^s=\sum_{k\le m} r_j^k \partial_s\partial_k \varphi$ so 
\[
  \langle T_p S, u_j^s: s\le m, j\le n \rangle \neq \R^d.
\]
But then there exists $a\neq 0$ in $(T_p S)^{\perp}$ orthogonal to
\[
 \sum_{s\le m} x_s u_j^s = \sum_{s,k\le m} x_s r_j^k \partial_s \partial_k \varphi 
\]
for every $x_s\in \R$ and $j\le n$. Thus
\[
 0=a\sum_{s,k\le m} x_s r_j^k \partial_s \partial_k \varphi =\sum_{s,k\le m} x_s r_j^k \partial_s \partial_k (a\varphi)=x^t H_a r_j
\]
with $x=(x_s)_{s\le m}$ and $r_j=(r_j^k)_{k\le m}$ in $\R^m$. This implies that
$
 H_a r_j=0 
$
for every $j\le n$. That $\dim V=n$ implies that $r_j$ are independent so $\dim\ker H_a\ge n$, and then (iii) is not satisfied. We can clearly reverse our reasoning to show that {\color{black} if (ii) is satisfied so is (iii).}

\end{proof}

As we said in the introduction, being totally curved rules out submanifolds of the type $S_1\times S_2$. Now we are going to prove that, and we begin by formalizing the kind of submanifolds that we want to avoid.

\begin{definition}[Primitive dimension]
\label{Primitive manifold} Let $S$ be a submanifold of $\R^d$ .  For any $p\in S$,  we define its \emph{primitive dimension at $p$} as the maximum $k\in\N$ for which the restriction of  $S$ near $p$ to any $k$ components of $\R^d$ is a manifold of dimension $k$.
\end{definition}

In terms of a parametrization of $S$, $\varphi=(\varphi_1,\ldots, \varphi_d)$, the primitive dimension is the maximal $k$ such that $\{\nabla \varphi_{i_j} \}_{j=1}^k$ are linearly independent for every $\{i_1,\ldots, i_k\}\subset \{1,\ldots ,d \}$. Clearly, if $S_1$ is a submanifold of $\R^k$ of dimension smaller than $k$, then the primitive dimension of $S=S_1\times S_2$ is less than $k$.

\begin{lemma}[Low primitive dimension implies smooth dependence of components]
\label{degeneracy}
Let $m\ge 2n \ge 2$. Let $S$ be a submanifold of $\R^d$ of dimension $m$,  codimension $n$ and   primitive dimension at most $m-n$ for every point in a neighborhood of $p$, with $S=\varphi((-1,1)^m)$, $\varphi:(-2,2)^m\to \R^d$ real-analytic and $\varphi(t)=(t,w(t))$. Then, perhaps after rearranging some components, for some $1\le h\le n$ we have
\[
\partial_j w_r=\sum_{i\le h-1} b_j^i \partial_i w_r  \qquad h\le j\le h+n-1
\]
for every $r\le h$, with $b_j^i$ real-analytic functions in a neighborhood of $p$.
\end{lemma}
\begin{proof}
 By definition of primitive dimension there exist $m-n+1$ components whose gradients are linearly dependent. Let us assume that exactly $h\le n$ of those components correspond to the last $n$ ones. Then, after rearranging them we can assume that
\[
 \dim \langle \partial \varphi_{n+h}, \partial\varphi_{n+h+1}, \ldots, \partial\varphi_{m+h}\rangle <m-n+1
\]
for some $1\le h\le n$ in a neighborhood of $p$.  Since $\varphi(t)=(t,w(t))$ we see that the matrix $(\partial \varphi_j)_{n\le j\le m}$ has the shape
\[
 \begin{pmatrix}
  0  &   D  \\
  I   &   *   \\
 \end{pmatrix}
\]
with $I$ the identity matrix of dimension $m+1-n-h$ and
\[
D=(\partial_j w_r)_{j\le n+h-1, r\le h}.
\]
Then  the rank of $D$ must be less than $h$. We can assume that the rank of $D$ equals a constant $0\le s<h$ in a neighborhood of $\varphi^{-1}(p)$. If $s=0$ then the result follows with $b_j^i=0$. Otherwise there exists some $s\times s$ non-zero minor and then, after possibly rearranging the components, the result follows from Cramer's rule.
\end{proof}

Now we shall prove that if $S$ is totally curved then its primitive dimension is larger than $m-n$.

\begin{proposition}[Totally curved implies large primitive dimension]
\label{curved_implies_primitive}
Let $S$ be totally curved at $p$, with $S=\varphi((-1,1)^m)$, $\varphi:(-2,2)^m\to\R^d$ real-analytic and $\varphi(t)=(t,w(t))$, where $m+n=d$. Then $S$ has primitive dimension larger than $m-n$ for every point in some neighborhood of $p$.
\end{proposition}
\begin{proof}
Suppose $S$ has primitive dimension at most $m-n$ for a sequence of points converging to $p$.  We must then prove that $S$ is not totally curved in a neighborhood of $p$. Since $S$ is real-analytic, looking into the proof of Lemma \ref{degeneracy} we see that its primitive dimension must be at most $m-n$ in a neighborhood of $p$. Thus by Lemma \ref{degeneracy} we have
\begin{equation}\label{EQ1}
\partial_j w_r=\sum_{i\le h-1} b_j^i \partial_i w_r  \qquad h\le j\le h+n-1 \quad r\le h.
\end{equation}
We consider the tangent fields
\begin{equation}\label{EQ2}
 v_j=\partial_j \varphi - \sum_{i\le h-1} b_j^i \partial_i \varphi        \qquad   h\le j\le h+n-1.
\end{equation}
We are going to see that the $v_j(p)$ generate a vector subspace of $T_p S$ of dimension $n$ and that
\begin{equation}\label{equ3}
 Q=\langle T_p S, \partial_s v_j: s\le m, h\le j\le h+n-1\rangle \neq \R^d,
\end{equation}
so by Proposition \ref{totally_curved_char} (ii) the result follows. We have
\[
\partial_l \varphi =(0,\ldots, 0, 1,0, \ldots ,0, \partial_l w_1, \ldots, \partial_l w_n)
\]
\[
 \partial_s\partial_l \varphi =(0, \ldots, 0, 0, 0, \ldots, 0, \partial_s\partial_l w_1,\ldots, \partial_s\partial_l w_n)
\]
for every $s,l\le m$, with the $1$ in the $l$th position. Then 
\[
 v_j=(-b_j^1,\ldots, -b_j^{h-1}, 0, \ldots ,0, 1,\ldots)
\]
with 1 in the $j$th position, which shows that $v_j$, $h\le n\le h+n-1$ are linearly independent. Differentiating we get that
\[
 \partial_s v_j= u_s^j- \sum_{i\le h-1} \partial_s b_j^i \partial_i\varphi
\]
with
\[
 u_s^j=\partial_s\partial_j\varphi - \sum_{i\le h-1} b_j^i \partial_s\partial_i \varphi 
\]
so
\[
Q=\langle T_p S, u_s^j: s\le m, h\le j\le h+n-1\rangle.
\]
By differentiating (\ref{EQ1}), for any $j,s$ we have
\[
 u_s^j-\sum_{i\le h-1} (\partial_s b_j^i) g_i \in 0\times \R^{n-h}
\]
with $g_i=(0,0,\ldots, 0, \partial_i w_1, \ldots, \partial_i w_h,0, \ldots 0)$. Thus
\[
 u_s^j\in \langle 0\times\R^{n-h}, g_i: i\le h-1 \rangle
\]
which is a vector space of dimension at most $n-1$, hence (\ref{equ3}) follows.

\end{proof}

\section{Singular case}

In this section we want to prove Theorem \ref{main_real} for singular functions. We start by defining a local version of $\lambda_S$. Let us fix a $C_c^{\infty}(\R^d)$ function $\psi$ with $\int_{\R^d} \psi=1$ and support contained in the unit ball of $\R^d$. For any $x_0\in \R^d$ and any $0<\beta<1/2$ we define the measure
\[
 \lambda_{S,x_0,\beta}(f)=\int_{\R^d} f(x) \frac{1}{\beta^m} \psi(\frac{x-x_0}{\beta}) \, d\lambda_S(x).
\]
By compactness of the support of $\lambda_S$ we have that $\lambda_{S,x_0,\beta}(\R^d)$ is bounded independently of $x_0$ and $\beta$.
If $\sigma_S$ is the volume measure on $S$ we have the following localization result.

\begin{proposition}[Localization in space]
\label{localization_in_space}
Let {\color{black}$0<\beta\le \delta<1$}.  {\color{black} There exists a constant $c>0$ (depending just on $S$) such that  for any} $f:\R^d\to \R$ with $\|f\|_{L^{\infty}}\le 1$ and $|\lambda_S(f)|>\delta$, the set of $x_1$ in $S$ for which there exists an $x_0\in \R^d$ at distance {\color{black} at most $2\beta$ with  $|\lambda_{S,x_0,\beta}(f)|>c\delta$} has $\sigma_S$-measure $\gg \delta$ (with the implicit constant depending just on $S$).
\end{proposition}
\begin{proof}
Since $\int_{\R^d}\psi =1$, by Fubini we have
\[
\int_{\R^d} \lambda_{S,x,\beta}(f) \, dx =\beta^n  \lambda_S(f).
\]
Let $G$ be the set of points $x_1$ in $\R^d$ for which there exists a point $x_0$ at distance smaller than $2\beta$ with $|\lambda_{S,x_0,\beta}(f)|>c\delta$.
Then, since the support of $g(x)=\lambda_{S,x,\beta}(f)$ is contained in a set of Lebesgue measure $O(\beta^n)$, we have
\[
\delta< |\lambda_S (f)|\ll \beta^{-n} m(G) +  \beta^{-n} c\delta \beta^n
\]
so for $c$ small enough we have
$
 m(G)\gg \delta \beta^n.
$
Now, $G\subset G'$, with $G'$ the set of points at distance at most $4\beta$ from $G\cap S$.  This is so because $\lambda_{S,x_0,\beta}(f)\neq 0$ implies that $x_0$ is at distance less than $ \beta$ from $S$. For the same reason $G$ is a union of balls of radius $ 2\beta$ and center at distance at most $\beta$ from $S$, so we can deduce that
\[
 m(G')\ll  \beta^n \sigma_S(G\cap S)
\]
which implies $\sigma_S(G\cap S)\gg \delta$ and the result follows.

\end{proof}

In what follows we will need to control the size of the set in which a real-analytic function is small. A particular case of Corollary 1 in \cite{garofalo_garrett} says that if $u$ is a real-analytic function in the ball {\color{black}$B_{1+2\delta_1}(0)\subset\R^k$ (with $\delta_1>0$ fixed) then there exist $p>1$ and $A>0$ such that}
\begin{equation}\label{garrofalo}
  \int_{B_{\delta_1/4}(y)} |u| \, dx  \, \, ( \int_{B_{\delta_1/4}(y)} |u|^{-1/(p-1)} \, dx)^{p-1} \le A
\end{equation}
for every $y\in B_1(0)$. From there we deduce

\begin{lemma}[Sublevel set estimate for analytic functions]
\label{sublevel_set_lemma}
Let $u$ be a real-analytic non-zero function on $B_{1+h}(0)\subset\R^k$ for some $h>0$. There exists $D=D(u)\ge 1$ such that for $0<\delta<1$ we have
 \[
  |\{ x\in B_1(0): |u(x)|<\delta\} |  \ll \delta^{1/D}.
 \]
\end{lemma}
\begin{proof}
We can cover $B_1(0)$ with a finite number of balls $B_j$ of radius $R=h/8$ centered inside $B_1(0)$. By applying (\ref{garrofalo}) with {\color{black}$\delta_1=h/2$} we get
\[
 \int_{B_j} |u| \, dx  \, \, ( \int_{B_j} |u|^{-1/(p-1)} \, dx)^{p-1} \le A
\]
for each of them. Since $u$ is non-zero, for every $j$ it cannot be zero for every point in $B_j$, and then $\int_{B_j} |u| \, dx\neq 0$ which implies
\[
 \int_{B_j}  |u|^{-1/(p-1)} \, dx \ll 1
\]
and then
\[
 \int_{B_1(0)} |u|^{-1/(p-1)} \, dx \ll 1
\]
from which the result follows with $D=p-1$.
\end{proof}

Finally we can prove our result in the singular case.

\begin{proposition}[Main result, singular case]
\label{main_singular}
Let $n\le d/4$.
For every $(T,\alpha)$-singular function $f$ we have
\[
 \lambda_S(f)\ll T^{-c}
\]
for some $c$ depending just on $\alpha$ and $S${\color{black}, and the implicit constant on the support of $\lambda_S$}.
\end{proposition}
\begin{proof}
We begin by picking $\delta=T^{-\epsilon}$ with $\epsilon>0$ a small constant.  Let us suppose that
$
 |\lambda_S(f)|\gg \delta;
$
we shall see that we arrive at a contradiction. By applying Proposition \ref{localization_in_space} we have
$
 |\lambda_{S,x_0,\beta}(f)| \gg \delta 
$
for some {\color{black}$x_0=x_0(x_1)$} at distance {\color{black}at most $2\beta$} from $x_1$, for every $x_1$ in a subset $G$ of $S$ of $\sigma_S$-measure $\gg \delta$. Since the support of $\lambda_S$ is compact, we can further assume that this set {\color{black}(more precisely, a subset of $G$ of $\sigma_S$-measure $\gg \delta$)} is contained in {\color{black}$\varphi((-1/2,1/2)^m)$}, with $\varphi:(-1,1)^m\to S$ some parametrization of $S$, {\color{black}with $\varphi(t)=(t,w(t))$, $w$} a real analytic function on $(-2,2)^m$.

{\color{black}Since we are assuming $\beta<\frac 14$, $x_1=(t_1,w(t_1))$ with $t_1\in (-1/2,1/2)^m$, $x_0=x_0(x_1)$ is at distance at most $2\beta$ from $x_1$,  and $\psi$ is supported on the unit ball , we conclude that}  for every $x_1\in G$ the support of $\psi_{x_0,\beta}(t)=\beta^{-m}\psi((\varphi(t)-x_0)/\beta)$ is contained in $(-1,1)^m$, and then we have
\[
 \lambda_{S,x_0,\beta}(f)=\int_{\R^m} f(\varphi(t)) \psi_{x_0,\beta}(t) R(t) \, dt,
\]
for $R(t)$ some $C_c^{\infty}$ function depending just on $\lambda_S$. Since $f$ is singular, for some $k\le d/2$ we can write {\color{black}$f(x',x'')=\tilde f(x')\tilde \rho(x'')$ for $(x',x'')\in\R^k\times \R^{d-k}$}. We have $R(t)=R(t_1)+O(\beta)$ for any $t$ in the support of $\psi_{x_0,\beta}$, with $t_1=\varphi^{-1}(x_1)$, and the same happens for $\tilde\rho(\varphi)$, so
\[
 \tilde\rho(x_1) \int_{\R^m} \tilde f(\varphi_1(t))\psi_{x_0,\beta}(t)\, dt\gg \delta,
\]
with $\varphi(t)=(\varphi_1(t),\varphi_2(t))\in\R^k\times\R^{d-k}$. 
On the other hand, since $d/2\le m-n$ and $S$ is totally curved by Proposition \ref{curved_implies_primitive} we have that its primitive dimension is larger than $k$ throughout some open subset $U$ of $\varphi((-1,1)^m)$.
Thus $\varphi_1(U)$ is a manifold of dimension $k$, which implies that  for some variables $t_{i_1},\ldots, t_{i_k}$, $i_j\le m$, the Jacobian $J(t)$ of $\varphi_1$ with respect to them is a nonzero real-analytic function on $(-2,2)^m$.  Assume they are the first $k$ variables. Let us write $t=(t',t'')$, $t'\in \R^k$, with $J=\det \partial\varphi_1/\partial t'$. 
By Lemma \ref{sublevel_set_lemma} there exists $D>1$ such that
\[
 |\{t\in(-1,1)^m: |J(t)|<\gamma^D\delta^D\}|\ll 
 \gamma\delta
\]
for any small constant $0<\gamma<1$. Since $\sigma_S(G)\gg \delta$ this means that for some $x_1\in G$ we have $|J(t_1)|\gg \delta^D$. Then, for $\beta\ll \delta^D$ we have $J(t)=J(t_1)+O(\beta)\neq 0$, so by the change of variables $(t',t'')\mapsto (\varphi_1,t'')$ we obtain that
 \[
  \tilde\rho(x_1)  \int_{\R^{m-k}} \int_{\R^k} \tilde f(\varphi_1) \psi_{\beta,x_0}(t) \,\frac{d\varphi_1}{|J(t)|} \, dt''\gg \delta
 \]
with $t=t(\varphi_1,t'')$ its inverse.  We have $J(t)=J(t_1) (1+O(\beta/\delta^D))$ in the support of $\psi_{\beta,x_0}$, so by picking $\beta= \delta^{2D}$ we can write
\[
 \tilde \rho(x_1) \int_{\R^k} \tilde f(\varphi_1) \tilde\psi_{\beta,x_0}(\varphi_1) \, d\varphi_1 \gg \delta^{D+1}
\]
with
\[
 \tilde\psi_{\beta,x_0}(\varphi_1)=\int_{\R^{m-k}} \psi_{\beta,x_0}(t) dt''.
\]
Since $\|\tilde\rho\|_{S^1}\le 1 \le  \int_{\R^{d-k}} \tilde \rho$, it follows that
\[
 \int_{\R^d} f(x) \eta(x) \,dx \gg \delta^{D+1}
\]
with $\eta(x_1,x_2)=\tilde\psi_{\beta,x_0}(x_1)$. Using implicit differentiation we can obtain the bound $\|\eta\|_{S^l}\ll \beta^{-2l}$ for every $l\ge 0$,  and since $f$ is {\color{black}$(T,\alpha)$-singular}  we have
\[
 T^{-\alpha} \beta^{-8d}\gg \delta^{D+1}
\]
which gives a contradiction for any $\epsilon<\alpha/18dD$, with $\delta=T^{-\epsilon}$, since $\beta=\delta^{2D}$ and $D>1$.

\end{proof}

\section{Mixing case: Low and high frequencies}

In the next sections we shall show that $\lambda_S(f)$ is small for $f$ a mixing function. Taking into account  our treatment of singular functions in the previous section, this will complete our proof of Theorem \ref{main}.

We will handle this case by using Fourier Analysis in $\R^d$. We begin by recalling the Fourier transform of the measure $\mu=\lambda_{S,x_0,\beta}$:
\[
 \widehat \mu (\xi)=\mu(e(-\xi\cdot))=\int_{\R^d} e(-\xi t) \, d\mu(t)   \qquad \xi \in \R^d.
\]
In this context we have Plancherel Theorem \cite[Theorem 7.1.14]{hormander},
\begin{equation}\label{plancherel}
 \int_{\R^d} f(x) \, d\mu(x) =  \int_{\R^d} \widehat f(\xi) \overline{\widehat\mu(\xi) }\, d\xi  \qquad  f\in C_c^{\infty}(\R^d)
\end{equation}
which can be seen as coming from Fourier expansion of $f$ plus linearity of $\mu$. Since $f$ will have a derivative of size $T$, it is convenient to split the frequencies into the following ranges
\begin{equation}\label{frequency_splitting}
 \mu=\mu^l+\mu^m+\mu^h
\end{equation}
where
$
 \mu^*(f)=\int_{\R^d} \widehat f(\xi) \overline{\widehat \mu(\xi)} \eta_*(\xi)\, d\xi
$,
for $*=l, m, h$ and
\[
 \eta_l(r)=\eta(\frac{r}{\rho T})   \qquad  \eta_m(r)=\eta(\frac{r}{T/\rho})-\eta(\frac{r}{\rho T})   \qquad   \eta_h(r)= 1-\eta(\frac{r}{T/\rho}),
\]
with $\rho\in (0,1)$ and  $\eta$ a fixed function in $C_c^{\infty}(\R^d)$ with $\eta(r)=1$ for any $|r|<1$ and $\eta(r)=0$ for any $|r|>2$. So $\mu^l$, $\mu^m$ and $\mu^h$ take care of the low, midrange and high frequencies respectively.

To handle $\lambda_{S,x_0,\beta}^l$ we are going to use the decay in average of $\widehat\lambda_{S,x_0,\beta}$, which is well known for the Fourier transform of a submanifold. This is a particular case of Theorem 7.1.26 in \cite{hormander}:
\begin{lemma}[$L^2$ decay of Fourier Transform of submanifold]
\label{fourier_l2_decay}
 For any $\beta<1$ and $K\ge 1$ we have
 \[
 \|1_{[1,2]}(|\cdot|)\widehat\lambda_{S,x_0,\beta}(K\cdot)\|_{L^2(\R^d)} \ll (\beta K)^{-m/2} .
\]
\end{lemma}
% 
% \begin{proof}
% We can assume that $S=\varphi((0,1)^m)$ with $\varphi(t)=(t,w(t))$. Then we can write
% \[
%  \lambda_{S,x_0,\beta}(f)=\int_{\R^m} f(t,w(t)) q(t) \,dt
% \]
% with $q(t)=\beta^{-m}\psi((\varphi(t)-x_0)/\beta) g(t)$ with $g(t)\ll 1$, the implicit constant depending just on $\lambda_S$. Consider the function
% \[
%  u(t,y)=q(t)\int_{\R^n} \tau_{K}(y-w(t)-r)\tau_{K}(r) \, dr=q(t) \tau_{K}*\tau_{K}(y-w(t))
% \]
% where $t\in\R^m, y\in \R^n$ and $\tau_{K}(r)=K\tau(Kr)$, $\tau$ a fixed function in $C_c^{\infty}(\R^n)$. We have
% \[
%  u(t,y)\ll \beta^{-m}K^{n}   \qquad \text{if}  \qquad |y-w(t)|\ll K^{-1}
% \]
% and $u(t,y)=0$ otherwise. Thus taking into account that $q(t)=0$ for $|t|\gg \beta$ we deduce that
% \begin{equation}\label{equ4}
%  \|u\|_{L^2(\R^d)}^2 \ll \beta^{-m}K^{n}.
% \end{equation}
% On the other hand
% \[
%  \widehat u(\xi_1,\xi_2)= \int_{\R^n}\int_{\R^m} q(t) \tau_{K}*\tau_{K}(y-w(t)) e(-\xi_1 t) e(-\xi_2 y) \, dt \, dy,
% \]
% so by the change of variable $y\mapsto y+w(t)$ we have
% \[
%  \widehat u(\xi_1,\xi_2)=\widehat\lambda_{S,x_0,\beta}(\xi_1,\xi_2)\widehat{\tau_{K}*\tau_{K}}(\xi_2).
% \]
% Moreover $\widehat{\tau_K*\tau_K}(\xi_2)=\widehat\tau_K^2(\xi_2)=\widehat\tau^2(\xi_2/K)$,
% so by picking $\tau$ such that its Fourier transform does not vanish on $1\le |\xi_2|\le 2$ we have
% \[
%  \int_{\R^d} |1_{[K,2K]}(|\xi|) \lambda_{S,x_0,\beta}(\xi)|^2\, \frac{d\xi}{K^d}\ll K^{-d}\int_{\R^d} |\widehat u(\xi)|^2 \, d\xi.
% \]
% Applying Plancherel identity $\|\widehat u\|_{L^2}=\|u\|_{L^2}$ and (\ref{equ4}) we are done.
% \end{proof}

Now the main result of this section is
\begin{proposition}[Mixing case, high and low frequencies control]
\label{main_low_high}
Let $f$ be a $(T,\alpha)$-mixing function. Then we have
\[
 |\lambda_{S,x_0,\beta}^l(f)|+|\lambda_{S,x_0,\beta}^h(f)|\ll (\rho^{2\alpha}/\beta^2)^{d/4} T^{n} + \rho
\]
for any $\rho>T^{-\alpha/2}$, where $\rho$ comes from the splitting (\ref{frequency_splitting}).
\end{proposition}
\begin{proof}
We can write
\[
 \lambda_{S,x_0,\beta}^h (f)= \lambda_{S,x_0,\beta}(f)-\int_{\R^d} \widehat{f}(\xi) \eta_{T/\rho}(\xi) \overline{\widehat\lambda_{S,x_0,\beta}(\xi)}\, d\xi
\]
with $\eta_{U}(r)=\eta(r/U)$. We can write
\[
 \widehat f(\xi)\eta_{T/\rho}(\xi)= \widehat f(\xi)\widehat{ \check{\eta}}_{T/\rho}(\xi) =\widehat{f*\check\eta_{T/\rho}}(\xi)
\]
where $*$ indicates the convolution and $\check\eta_U(x)$ is the inverse Fourier Transform of $\eta_U$. By (\ref{plancherel}) we have
 \[
  \lambda_{S,x_0,\beta}^h (f)= \lambda_{S,x_0,\beta}(f-f*\check\eta_{T/\rho}).
 \]
By definition $1=\eta_U(0)=\int_{\R^d} \check\eta_U(y)\, dy$ and then
\[
 f(x)-f*\check\eta_U(x)=\int_{\R^d} [f(x)-f(x-y)] \check\eta_U(y) \, dy.
\]
Since $|f(x)-f(x-y)|\le \|f\|_{S^1} |y|\le T|y|$ and  $\check\eta_U(x)=U^d \check\eta(Ux)$, by a change of variables we have
\[
| f(x)-f*\check\eta_U(x)| \le \frac TU \int_{\R^d} |y| |\check\eta(y)| \, dy \ll T/U
\]
hence
$
 \lambda_{S,x_0,\beta}^h(f)\ll \rho.
$
On the other hand, since $\eta$ is supported on the ball of radius 2 we have
\[
 \lambda_{S,x_0,\beta}^l(f)=\int_{\R^d} \widehat f(\xi)\eta_{\rho T}(\xi) \overline{\widehat\lambda_{S,x_0,\beta}(\xi)} 1_{[0,2\rho T]}(|\xi|) \, d\xi
\]
and then reasoning as before and using Cauchy-Schwarz we have
\[
 |\lambda_{S,x_0,\beta}^l(f)|\le \| \widehat{f*\check\eta_{\rho T}}\|_{L^ 2} \|\widehat\lambda_{S,x_0,\beta} 1_{[0,2\rho T]}(|\cdot|)\|_{L^2}
\]
so by Lemma \ref{fourier_l2_decay} and Plancherel we have
\[
 |\lambda_{S,x_0,\beta}^l(f)|\ll \beta^{-m/2} (\rho T)^{n/2} \|f*\check\eta_{\rho T}\|_{L^2}.
\]
Moreover
\[
 \|f*\check\eta_{\rho T}\|_{L^2}^2=\int_{\R^d\times \R^d}   \check\eta_{\rho T}(y)\overline{\check\eta}_{\rho T}(y') \int_{\R^d} f(x-y) \overline f(x-y')\, dx    \, dy\, dy'
\]
so by using the properties of the Fourier Transform and changing variables we have
\[
 \|f*\check\eta_{\rho T}\|_{L^2}^2=\int_{\R^d} w(y) \int_{\R^d} f(x-\frac{y}{T}) \overline f(x) \, dx \, dy
\]
with
\[
 w(y)=\rho^d \int_{\R^d} \check\eta(\rho y+y') \overline{\check\eta}(y')\, dy'.
\]
By the decay of the Fourier transform we have the inequality $w(y)\ll \rho^d \min(1,|\rho y|^{-1/|o(1)|})$ so that applying (\ref{horosphere_condition}) for $|y|<\rho^{-1}\epsilon^{-1}$ and the bound  $\|f\|_{S^0}\le 1$ and the compact support of $f$ we have
\[
 \|f*\check\eta_{\rho T}\|_{L^2}^2\ll \rho^d (\rho^{-1}\epsilon^{-1})^{d(1-\alpha)}+ \epsilon^{1/|o(1)|}\ll \epsilon^{-d}\rho^{\alpha d}+\epsilon^{1/o(1)},
\]
and choosing $\epsilon=\rho^{n/d}$ we get that $\lambda_{S,x_0,\beta}^l(f)\ll \beta^{-m/2} \rho^{\alpha d/2} T^{n/2}$.

\end{proof}

\section{ A Sublevel set estimate}

In order to control the midrange frequencies in the decomposition (\ref{frequency_splitting}), we will use both the shape and the decay of the Fourier transform of $\lambda_{S,x_0,\beta}$. This Fourier transform can behave badly if $S$ is not totally curved at $x_0$, or if $x_0$ is near to such a point.
In order to distinguish the points for which the Fourier transform behaves nicely we are going to quantify the concept of being totally curved.

Let $p\in S$ and fix  a real-analytic parametrization $\varphi=\varphi_p:U\to S$ in $S$ with $\varphi^{-1}(p)\in U\subset \R^m$.  For every $a\in (T_pS)^{\perp}$ we consider $H_a$, the Hessian of $a \varphi$. $H_a$ has $m$ eigenvalues (counted with multiplicities), and if we order their absolute values we get a sequence
$
 0\le \beta_1\le \beta_2\le \ldots \le \beta_m.
$
We define the functions $e_{a,\varphi}(p)=\beta_n$  and
\begin{equation}\label{eigenvalue_n_hessian}
 e_S(p)=\inf\{e_{a,\varphi}(p): a\in (T_pS)^{\perp}, |a|=1\}.
 \end{equation}
With them we give the following definition.
\begin{definition}[$\delta$-curved point]
\label{curved points}
Let $0<\delta<1$. We say that $p\in S$ is a \emph{$\delta$-curved point} if $e_S(p)>\delta$.
\end{definition}
\begin{remark}
The fact that $p$ is a $\delta$-curved point actually depends on the chosen parametrization $\varphi_p$. Afterwards this will not cause any problem since by compactness we will just need a finite number of parametrizations to cover the support of $\lambda_S$.
\end{remark}

By Proposition \ref{totally_curved_char}, we have that $S$ is not totally curved at $p$ if and only if $p$ is a $0$-curved point. We shall be able to control $\lambda_{S,x_0,\beta}^{m}$ for $x_0$ a $\delta$-curved point for a suitable $\delta$. Thus, we need to show that the set of points that are not $\delta$-curved is negligible, namely
\[
\sigma_S(\{p\in \supp \lambda_S: e_S(p)<\delta\})<\delta^{1/c}  
\]
for some $c>0$.
This is a sublevel set estimate for the function $e_S$; this kind of estimates play an important role in some problems of Fourier Analysis \cite{carbery}. Our case is quite peculiar due to the fact that \emph{every} $a\in (T_p S)^{\perp}$ is involved in the definition of $e_S(p)$.
 In order to prove such an estimate we will use that $S$ is totally curved and the \emph{Real Nullstellensatz}, that characterizes when several polynomials have the same real root:

Let $lst=(n,d_1,d_2,\ldots, d_n)$ be a list of positive integers; define $m_i$ to be $\binom{n+d_i-1}{n-1}$ and
\[
 m''=m_1+\ldots +m_n.
\]
We consider $f_1(c,z), \ldots , f_n(c,z)$ the generic homogeneous polynomials in the variables $z=(z_1,\ldots ,z_n)$ of degrees $d_1,\ldots ,d_n$, with $c$ the generic coefficients, $c\in \R^{m''}$. We define the set $W_{lst}(\R)$ of $c\in \R^{m''}$ such that the system of equations
\begin{equation}\label{equ6}
 f_1(c,z)=f_2(c,z)=\ldots =f_n(c,z)=0
\end{equation}
has no solution  $z\in\R^n$, $z\neq 0$. We have that $W_{lst}(\R)$ is a semialgebraic set and there exists \cite[Theorem C]{gonzalez_lombardi} an algebraic identity
\begin{equation}\label{positivestellensatz}
 p_1(c) |z|^{2s}+\sum_{j\ge 2} p_j(c) a_j(c,z)^2 + \sum_{i=1}^n f_i(c,z) b_i(c,z)^2 = 0,
\end{equation}
with $s$ a positive integer, $p_j$ semipolynomial functions, and $a_j, b_i$ polynomials in $z$ with coefficients semipolynomial functions, such that $p_1>0$ and $p_j\ge 0$  on $W_{lst}(\R)$ for every $j\ge 2$ (a semipolynomial function is a function built using polynomials and the absolute value several times). 

So $p_1$ is a kind of real resultant, because if there is no non-zero solution for the system (\ref{equ6}) then $p_1>0$.

\begin{proposition}[Non $\delta$-curved points are negligible]
\label{sublevel}
Let $S$ be a totally curved submanifold of $\R^d$ of dimension $m$. Then there exists $c_S>1$ such that
the $\sigma_S$ measure of set of points $p$ in the support of $\lambda_S$ which are not $\delta$-curved is $O_{\lambda_S}(\delta^{1/c_S})$.
\end{proposition}
\begin{proof}
By compactness of $\supp \lambda_S$ it is enough to assume that $S=\varphi((-1,1)^m)$ with $\varphi(t)=(t,w(t))$ and $w:(-2,2)^m\to \R^n$ a real-analytic function and that $\varphi$ is the parametrization chosen to define $e_S(p)$ for every $p$. Moreover we can substitute $e_S(p)$ by 
\[
 e_S^*(p)=\inf\{e_{(0,z),\varphi}(p): z\in\R^n, |z|=1\}
\]
The characteristic polynomial of the matrix $H_{(0,z)}(t)$ can be written as
\[
 \lambda^m+s_{m-1}\lambda^{m-1}+ \ldots + s_1 \lambda +s_0.
\]
Now, $s_j$ is a homogeneous polynomial of degree $m-j$ in $z$, and then we can write it as
\[
 s_{j-1}(t,z)=f_j(\varrho(t),z)     \qquad 1\le j\le m
\]
with $f_1(c,z), \ldots, f_n(c,z)$ the generic homogeneous polynomials of degrees $m, m-1,\ldots, m-n+1,$ each component of $\varrho$ being a polynomial in $ \partial_i\partial_j w(t)$, so a real-analytic function. By (\ref{positivestellensatz}), for every $t\in (-1,1)^m$ we have
\begin{equation}\label{resultant_equation}
 p_1(\varrho(t)) |z|^{2s}+\sum_{j\ge 2} p_j(\varrho(t)) a_j(\varrho(t),z)^2 + \sum_{i=1}^n s_{i-1}(t,z) b_i(\varrho(t),z)^2 = 0.
\end{equation}
We can write
\[
 (-1,1)^m=\varrho^{-1}(W)\cup \varrho^{-1}(W^c),
\]
with $W=W_{lst}(\R)$ for $lst=(n,m,m-1,\ldots,m-n+1)$.
If $t\in \varrho^{-1}(W^c)$ then the dimension of $\ker H_z$ is at least $n$.  Hence, since $S$ is totally curved, Proposition \ref{totally_curved_char} implies that there cannot exist an open set $U\subset \varrho^{-1}(W^c)$. But writing $\varrho(v)=\Phi(D^2 w(v))$, with $\Phi$ a polynomial, we have that $\Phi^{-1}(W)$ and $\Phi^{-1}(W^c)$ are semialgebraic sets and
\[
 \varrho^{-1}(W^c)= \{t\in (-1,1)^m : D^2 w\in \Phi^{-1}(W^c)\}=
 \]
\[ 
 =\bigcup_j \{t: q_j(D^2 w(t))>0\} \cup \bigcup_k \{t: r_k(D^2 w(t))=0\}
\]
with $q_j$ and $r_k$ polynomials, and since the set concerning the $q_j$ is open it must be void, and then
\[
  \varrho^{-1}(W^c)=\bigcup_k \{t: r_k(D^2 w(t))=0\},
\]
with $r_k(D^2 w)$ non-zero as a real-analytic function. Thus $\varrho^{-1}(W^c)$ is a set of measure zero. On the other hand, since $p_1$ is a semipolynomial and $p_1>0$ on $\varrho^{-1}(W)$, we can split $\varrho^{-1}(W)$ into a disjoint and finite union of sets $V_j$, and on each of them we have $p_j(\Phi(D^2 w(t)))=h_j(D^2 w(t))$, with $h_j$ a polynomial. Clearly $h_j(D^2 w)$ can be extended to a real analytic non-zero function on $(-2,2)^m$ and we have
\[
m \{t\in(-1,1)^m: p_1(\varrho(t))<\delta \} \le \sum_j m\{t\in(-1,1)^m: |h_j(D^2 w)(t)|<\delta\}.
\]
Then, by applying Lemma \ref{sublevel_set_lemma} to the functions $h_j(D^2 w)$ we have
\begin{equation}\label{equ5}
 m(\{t: p_1(\varrho(t))<\delta \})<\delta^{1/c}.
\end{equation}
Let us consider a fixed $t\in \varrho^{-1}(W)$ with $p_1(\varrho(t))\ge \delta$.  Since $p_j(\rho(t))\ge 0$ for every $j\le m$, by (\ref{resultant_equation}) for every $|z|=1$ we have 
\[
 -s_{i-1}(t,z)\gg \delta
\]
for some $i\le n$. But by the definition of $s_{i-1}(t,z)$ as a symmetric polynomial in terms of the roots of the characteristic polynomial, we get that
\[
| \lambda_1\lambda_2\ldots \lambda_{m-(i-1)}| \gg \delta 
\]
for some $\lambda_j$, $j\le m-(i-1)$ eigenvalues of $H_{(0,z)}$. But, since all eigenvalues are $O(1)$ we get that
\[
 |\lambda_j|\gg \delta
\]
for every $j\le m-(i-1)$. Thus if $t\in\varrho^{-1}(W)$ with $p_1(\varrho(t))\ge \delta$ then $e_S^*(\varphi(t))\gg \delta$,  so the result follows from (\ref{equ5}).
\end{proof}

This result implies that it is enough to control $\lambda_{S,x_0,\beta}(f)$ for $x_0$ near curved points in order to control $\lambda_S(f)$.

\begin{proposition}[Bound at curved points implies main result]
\label{local_to_global}
Let $f:\R^d\to\R$ with $\|f\|_{L^{\infty}}\le 1$.
If $\lambda_{S,x_0,\beta}(f)\ll \epsilon$ for every point $x_0$ at distance {\color{black} at most $2\beta$} from some $\gamma$-curved point of $S$, then
\[
 \lambda_S(f)\ll \epsilon+\beta+\gamma^{1/c_S},
\]
with $c_S>0$ a constant depending on $S$, {\color{black} and the implicit constant depending on the support of $\lambda_S$.}
\end{proposition}
\begin{proof}
We can assume that $\epsilon\ge \beta$, because otherwise we could {\color{black}replace} $\epsilon$ by $\beta$. By Proposition \ref{sublevel} we have  $\lambda_{S,x_0,\beta}(f)\ll \epsilon$ for every point at distance {\color{black} at most $2\beta$} from a subset $E\subset \supp\lambda_S$, with $\sigma_S(\supp\lambda_S\setminus E)\ll \gamma^{1/c_S}$. Thus if we were to have $\lambda_S(f)\gg \delta$, with $\delta=c(\epsilon+\beta+\gamma^{1/c_S})$, $c$ a large constant, we would get a contradiction by applying Proposition~\ref{localization_in_space}.
\end{proof}

\section{Fourier transform of the submanifold}

Here we want to study $\widehat\lambda_{S,x_0,\beta}$ at each frequency $\xi$, when $x_0$ is at distance $O(\beta)$ from a $\beta$-curved point on $S$. We would like to proceed by performing stationary phase, but we have the problem that we do not control all the eigenvalues of the Hessian of the phase function in the oscillatory integral. By the properties of the point $x_0$ however, we know that most eigenvalues are large, so we shall first separate the small eigenvalues and then use stationary phase with the big ones. We begin by diagonalizing a quadratic form with non-constant coefficients in the case in which the diagonal dominates over the rest of coefficients.

\begin{lemma}[Analytic diagonalization]
\label{diagonalization}
Let
\[
 F(x)=\sum_{i\le l} \lambda_i x_i^2 +2\delta \gamma\sum_{i\le j\le l} x_i x_j  \phi_{ij}(x)
\]
with $0<\delta<1$, $\phi_{ij}(x)$ real analytic functions on $[0,1]^l$, and $\gamma\le |\lambda_l|\le \ldots \le |\lambda_1|$, $\lambda_j\in\R$. For $\delta$ small enough (depending just on the functions $\phi_{ij}$)  there is a real-analytic change of variable $y=\psi(x)$ such that
\[
 F(x)=\sum_{i\le l} \lambda_i y_i^2
\]
and $\det D\psi(x)=1+O(\delta)$.
\end{lemma}
\begin{proof}
We can assume $\gamma=1$. We begin by $x_1$, writing 
\[
 F(x)=(\lambda_1+2 \delta\phi_{11}(x))x_1^2 + 2  x_1 \delta \sum_{1<j}  x_j\phi_{1j}(x) +\ldots 
\]
and then
\[
 F(x)= (\lambda_1+2\delta \phi_{11}(x)) [ x_1 +\delta\frac{ \sum_{1<j} x_j \phi_{1j}(x)}{\lambda_1+2\delta \phi_{11}(x)}]^2 - \frac{ (\delta\sum_{1<j}x_j \phi_{1j}(x))^2}{\lambda_1+2 \delta\phi_{11}(x)}+\ldots
\]
so we start with the change
\[
 x_1^*=x_1 +\delta\frac{ \sum_{1<j} x_j \phi_{1j}(x)}{\lambda_1+2\delta \phi_{11}(x)}
\]
that for $\delta$ small enough satisfies
\[
 \frac{\partial x_1^*}{\partial x_1}= 1 + O(\delta), \qquad   \frac{\partial x_1^*}{\partial x_j}=  O(\delta), 
\]
for any $j\ge 2$. Thus 
\[
 F=(\lambda_1+\delta \phi_{11}(x)) (x_1^*)^2 + \sum_{2\le i\le l}\lambda_i x_i^2+ 2\delta\sum_{2\le i\le j\le l} x_i x_j \phi_{ij}^*(x)
\]
with 
\[
\phi_{ij}^*=\phi_{ij} - \frac{\delta^2}{\lambda_1+2\delta \phi_{11}(x)}  \phi_{1i} \phi_{1j}.
\]
Continuing like that with $x_2,x_3,\ldots$, we arrive at
\[
 F=\sum_i (\lambda_i + \delta f_i(x^*)) (x_i^*)^2 =\sum_i \lambda_i(1 + \frac{\delta}{\lambda_i} f_i(x_*)) (x_i^*)^2
\]
with $\partial x_i^*/\partial x_i=1+O(\delta)$ and $\partial x_i^*/\partial x_j= O(\delta)$ for any $i\neq j$, and $f_i(x^*)$ are real-analytic functions on $[0,1]^l$ with bounded derivatives.
We finish with the change
\[
 y_i= x_i^*\sqrt{1+\frac{\delta}{\lambda_i} f_i(x^*)}= x_i^* + \frac{\delta}{\lambda_i} \tilde f_i(x^*)
\]
whose determinant is a smooth function near 1, and then
\[
 F=\sum_i \lambda_i  y_i^2
\]
with $\det Dy/Dx=1+O(\delta)$.
 
\end{proof}

{\color{black}
In order to study the Fourier transform of the submanifold we will use the following result.

\begin{lemma}[Simple Stationary Phase]
\label{fourier_integral}
Let $g(t,x)\in C_c^{\infty}(\R\times\R^l)$ be a function supported on the ball of radius $r>0$ of $\mathbb R^{l+1}$, then for any $\lambda>1$ we can write
\[
 \int_{\R} g(t,x) e(\lambda t^2)\, dt = \lambda^{-1/2} G_{\lambda}(x)
\]
where $G_{\lambda}\in C_c^{\infty}(\R^l)$ is a function supported on the ball of radius $r$ in $\mathbb R^l$ and satisfying
\[
 \|G_{\lambda}\|_{L^{\infty}}\le 5 \|g\|_{L^{\infty}}+ 5\|\partial_t g\|_{L^{\infty}}+5\|\partial_t^2 g\|_{L^{\infty}}.
\]
\end{lemma}
\begin{proof}
We have that
\[
 G_{\lambda}(x)=\sqrt{\lambda}\int_{\mathbb R} g(t,x)e(\lambda t^2)\, dt.
\]
We see that the question about the support is trivial, hence we just need to bound $\|G_{\lambda}\|_{L^{\infty}}$. By performing the change $t\mapsto t/\sqrt{\lambda}$ we get that
\[
 G_{\lambda}(x)=\int_{\mathbb R} g(\frac{t}{\sqrt{\lambda}},x)e(t^2)\, dt.
\]
By splitting the integral we have $|G_{\lambda}(x)-G_{\lambda}^{+}(x)|\le 2\|g\|_{L^{\infty}}$ with
\[
  G_{\lambda}^{+}(x)=\int_{1}^{\infty} [g(\frac t{\sqrt \lambda}, x)+g(-\frac{t}{\sqrt{\lambda}},x)]e(t^2) \, dt,
\]
and by a change of variables
\[
 G_{\lambda}^{+}(x)=\int_{1}^{\infty} \frac{g(\frac{\sqrt{s}}{\sqrt \lambda}, x)+g(-\frac{\sqrt{s}}{\sqrt{\lambda}},x)}{2\sqrt{s}}e(s) \, ds.
\]
Now, by integrating by parts twice we have
\[
 |\int_1^{\infty} h(s) e(s)\, ds|\le \|h\|_{L^{\infty}}+\|h'\|_{L^{\infty}}+\|h''\|_{L^1}
\]
for any $h\in C_c^{\infty}([1,\infty))$, the norms considered in that interval. Applying it to the integral defining $G_{\lambda}^{+}(x)$ we get that
\[
 \|G_{\lambda}^{+}\|_{L^{\infty}}\le 3\|g\|_{L^{\infty}}+3\|\partial_t g\|_{L^{\infty}}+3\|\partial_t^2 g\|_{L^{\infty}}
\]
so the result follows.
\end{proof}

}

Now we give our result for the Fourier transform of the manifold, in which we can see the resemblance with the case of the sphere.

\begin{proposition}[Fourier Transform of the submanifold]
\label{stationary_phase}
Let $\xi_0\in \R^d$ with $|\xi_0|>1$.  There exist constants $c'<1<c$ depending just on $S$ such that if $x_0\in \R^d$ is a point at distance $O(\beta)$ from a $c\beta$-curved point of $S$, with $0<\beta<c'$, then for every $\xi\in\R^d$ with $|\xi-\xi_0|<\beta|\xi_0|$ we have
\[
 \widehat\lambda_{S,x_0,\beta}(\xi)=(|\xi_0|\beta^3)^{-\frac{m-n}2}\int_{\R^n} \Phi_r(\tilde\xi)e(|\xi_0|\Phi_r^*(\tilde\xi))\, dr +  O((|\xi_0|\beta^3)^{-\frac 1{|o(1)|}}),
\]
with $\tilde\xi=\xi/|\xi_0|$, $\Phi_r,\Phi_r^*$ depending on $S, x_0$ and $\beta$ but with uniformly bounded derivatives and $\Phi_{\cdot}(\cdot)$ with uniformly bounded  support on $\R^{2d}$.
\end{proposition}

\begin{proof}
We can assume that  near $x_0$ the manifold $S$ is parametrized as $(t,w(t))$, with  $w:(-1,1)^m\to \R^n$. Then $ \widehat\lambda_{S,x_0,\beta}(\xi)$ can be written as
\[
 \widehat\lambda_{S,x_0,\beta}(\xi)=\beta^{-m}\int_{\R^m} e(-\xi_1 t-\xi_2 w(t)) \, \psi_*(\frac{t-t_0}{\beta}) \, dt
\]
with $\xi=(\xi_1,\xi_2)$ and $\psi_*(s)=\psi(s+c_1, \frac{w(t_0+\beta s)-w(t_0)}{\beta}+c_2) g(t_0+\beta s)$, where $(t_0,w(t_0))$ is the $c\beta$-curved point on $S$ at distance $O(\beta)$ from $x_0$, $c_1,c_2$ bounded constants, and $g$ a $C_c^{\infty}$ function coming from $\lambda_S$. Thus we have that $\psi_*$ is $C^{\infty}$ function with uniformly bounded support and derivatives. By the change $t\mapsto \beta(t+t_0)$ we have
\[
 \widehat\lambda_{S,x_0,\beta}(\xi)=\int_{\R^m}    \psi_*(t) e(|\xi_0|(-\tilde\xi_1 (t_0+\beta t)-\tilde\xi_2 w(t_0+\beta t))) \, \, dt,
\]
with $\tilde\xi=(\tilde\xi_1,\tilde\xi_2)$. Let $\tilde\xi_0=(\tilde\xi_{0,1},\tilde\xi_{0,2})$. If $|\tilde\xi_{0,2}|<k$ for some constant $0<k<1$ depending just on $w$, integrating by parts with respect to some variable we have $\widehat\lambda_{S,x_0,\beta}\ll (\beta|\xi_0|)^{-1/|o(1)|}$, hence we can assume that $|\tilde\xi_{0,2}|\gg 1$. Let $g_0$ be an orthogonal matrix such that the matrix $g_0^t D^2(\tilde\xi_{0,2}w)(t_0)g_0$ is diagonal. Then since $(t_0,w(t_0))$ is a $c\beta$-curved point on $S$, by the change $t\mapsto g_0 t$ we get that
\[
 \widehat\lambda_{S,x_0,\beta}(\xi)=e_{\xi}\int_{\R^m} \tilde\psi(t) e(|\xi_0|\sum_{j\le m} {\color{black}(}\lambda_j \beta^2 t_j^2-s_j\beta t_j{\color{black})}) e(|\xi_0|\beta^3 W(\tilde\xi,t)) \, dt
\]
with {\color{black}$e_{\xi}=e(-|\xi_0|(t_0\tilde\xi_1+w(t_0)\tilde\xi_2))$ and $\tilde\psi(t)=\psi_*(g_0 t)$}, $\lambda=(\lambda_j)_{j\le m}$ constant, $\lambda=O(1)$, with $|\lambda_j|\ge c\beta$ for every $j\le m-n$, $W$ real-analytic on $(-2,2)^d\times\supp\tilde \psi$ with uniformly bounded coefficients and terms of degree at least 2 in $t$, and $(s_j)_{j\le m}=(\tilde\xi_1+\tilde\xi_2 Dw(t_0) )g_0$. By separating the variables $r=(t_{m-n+1},\ldots t_m)$ we can write
\[
\widehat\lambda_{S,x_0,\beta}(\xi)=\int_{\R^n} e(|\xi_0|\Phi_{r,1}^*(\tilde\xi)) \, \int_{\R^{m-n}} \psi_r(t) e(|\xi_0|f_r(\tilde\xi, t)) \, dt \, dr
\]
 with $\psi_r$ and $\Phi_{r,1}^*$ satisfying the same properties as $\tilde\psi$ and $\Phi_r^*$ respectively, 
\[ 
f_r(\tilde\xi,t)=\beta^3 W_r(\tilde\xi,t)+\sum_{j\le m-n} \lambda_j \beta^2 t_j^2-s_j \beta t_j.
\]
If  $|s_j|\gg \beta|\lambda_j|$ for some $j\le m-n$ then integrating by parts we have $\widehat\lambda_{S,x_0,\beta}(\xi)\ll (\beta^3|\xi_0|)^{-1/|o(1)|}$. Thus, we can assume $s_j/\beta\lambda_j\ll 1$ for every $j\le m-n$.  The equation $Df_r(\tilde\xi, t)=0$ can be written as
\[
  t_j= \frac{s_j(\tilde\xi)}{2\lambda_j\beta}-\frac{\beta}{2\lambda_j} \frac{\partial W_{r}}{\partial t_j}(\tilde\xi,t)   \qquad  j\le m-n;
\]
we have $|\beta/2\lambda_j|\le 1/2c$, and then for $c$ large enough that equation has a unique solution $t=\varrho_r(\tilde\xi)$, with $\varrho_r$ a real-analytic function with uniformly bounded coefficients on the support of $\psi_r$. By the change $t\mapsto t+\varrho_r(\tilde\xi)$ we have
\[
\widehat\lambda_{S,x_0,\beta}(\xi)=\int_{\R^n} e(|\xi_0|\Phi_{r,2}(\tilde\xi)) \, \int_{\R^{m-n}} \psi_{r,2}(t,\tilde\xi) e(|\xi_0|f_r^*(\tilde\xi,t)) \, dt \, dr
\]
with
\[
f_r^*(\tilde\xi,t)= \beta^3 W_r^*(\tilde\xi,t)+\sum_{j\le m-n}\lambda_j \beta^2 t_j^2,
\]
$W_r^*$ real-analytic with uniformly bounded coefficients and terms of degree at least 2 in $t$.
By applying the change from Lemma \ref{diagonalization} we have
\[
\widehat\lambda_{S,x_0,\beta}(\xi)=\int_{\R^n} e(|\xi_0|\Phi_{r,3}(\tilde\xi)) \, \int_{\R^{m-n}} \psi_{r,3}(t,\tilde\xi) e(\beta^2|\xi_0|\sum_{j\le m-n} \lambda_j t_j^2) \, dt \, dr.
\]
{\color{black}
We can assume that $c\beta^3 |\xi_0|>1$, because otherwise the statement of the proposition is trivial. Then, the result will follow by applying Lemma \ref{fourier_integral} iteratively in order to handle the the inner integral. For instance, in the first iteration we would get that
\[
 \int_{\mathbb R} \psi_{r,3}(t_1,t',\tilde{\xi})e(\beta^2|\xi_0|\lambda_1 t_1^2)\, dt_1=(\beta^2 |\xi_0|\lambda_1)^{-1/2}\psi_{r,4}(t',\tilde{\xi})
\]
with $t'=(t_2,t_3,\ldots,t_{m-n})$ and $\psi_{r,4}$ a function with bounded derivatives and support. In the final iteration we would arrive at
\[
\widehat\lambda_{S,x_0,\beta}(\xi)=\int_{\R^n} e(|\xi_0|\Phi_{r,3}(\tilde\xi))  \frac{\psi_{r,3+m-n}(\tilde\xi)}{\sqrt{\prod_{j\le m-n} \beta^2|\xi_0|\lambda_j}} \, dr,
\]
and the result follows by using that $\lambda_j^{-1}=\beta^{-1}(\beta/\lambda_j)$ with $\beta/\lambda_j=O(1)$.
}
\end{proof}

\section{Mixing case: Midrange frequencies}

We begin by noticing what we need in order to control the midrange frequencies, after using the information about the Fourier Transform of $S$ obtained in the previous section.

\begin{proposition}[Oscillatory integral controls midrange frequencies]
\label{midrange_bound_oscillatory}
Let $S$ be totally curved and $f\in C_c^{\infty}(\R^d)$. Let $x_0$ be a point at distance $O(\beta)$ from a $c\beta$-curved point on $S$, with $0<\beta<1$.  There exist $\Phi,\Phi^*\in C_c^{\infty}(\R^d)$ with uniformly bounded derivatives and support, and $\rho T\ll U \ll T/\rho$ such that
\[
 \lambda_{S,x_0,\beta}^m(f)\ll U^n \beta^{-3d} \log(1/\rho) [ I_U + O((\beta^3 U)^{-1/|o(1)|})]
\]
with
\[
 I_U=\frac{1}{\sqrt{ U^d}}\int_{\R^d} \Phi(\frac{\xi}{U}) \widehat f(\xi) e(x\xi) e(U\Phi^*(\frac{\xi}{U})) \, d\xi.
\]
for some $x\in\R^d$.
\end{proposition}
\begin{proof}
We can split the frequencies into $O(\beta^{-d}\log(1/\rho))$ pieces, so that
\[
 \lambda_{S,x_0,\beta}^m(f)\ll \beta^{-d}\log(1/\rho) \int_{\R^d} \eta_*(\frac{\xi-\xi_0}{\beta |\xi_0|})  \widehat f(\xi)\widehat\lambda_{S,x_0,\beta}(\xi) \, d\xi
\]
with $\eta_* \in C_c^{\infty}$ independent from $\beta,\rho$ and $T$, $\xi_0$ some frequency in the range $\rho T\ll |\xi_0|\ll T/\rho$.  By applying Proposition \ref{stationary_phase}, interchanging the integrals and using the compactness of the support of $\Phi_{\cdot}(\cdot)$ we have
\[
 \lambda_{S,x_0,\beta}^m(f)\ll U^n \beta^{-3d} \log(1/\rho) [ \tilde I_U + O((\beta^3 U)^{-1/|o(1)|})]
\]
with $U=|\xi_0|$ and
\[
 \tilde I_U=\frac{1}{\sqrt{ U^d}}\int_{\R^d} \eta_*(\frac{\xi-\xi_0}{\beta U}) \Phi(\frac{\xi}{U})  \widehat f(\xi) e(U\Phi^*(\frac{\xi}{U})) \, d\xi.
\]
Now, by using the formula $\eta_*(v)=\int_{\R^d} \widehat\eta_*(x)e(vx)\, dx$, interchanging the order of the integrals and applying the trivial H\"older inequality we are done, since $\int_{\R^d} |\widehat\eta_*(x)|\, dx<\infty$.

\end{proof}

We could try to control the oscillatory integral appearing in the statement of the previous proposition by applying the Fourier Transform (Plancherel Theorem), but in doing so we would come back essentially to our original definition for $\lambda_{S,x_0,\beta}$, and then we would get just the trivial bound $O(1)$.
This problem comes from the fact that the range of frequencies $\xi$  and  the growth of  the exponential's phase are both very large; we are going to see that both problems can be addressed by using Weyl-Van der Corput inequality before applying the Fourier Transform, and we just have to pay  by having to control $\widehat f_s$ instead of $\widehat f$, with $f_s(x)=f(x+s)\overline f(x)$.

\begin{proposition}[Frequency control]
\label{van_der_corput}
Let $f\in C_c^{\infty}(\R^d)$ with $\|f\|_{L^{\infty}}\le 1$ {\color{black} and supported on a translate of the unit ball}.
Let $1\le U$ and $\Phi^{**}(\xi)=\Phi^*(\xi)+y\xi$ with $y\in\R^d$ and $\Phi,\Phi^*\in C_c^{\infty}(\R^d)$.  For every ${\color{black}1\le}V\le U$ and $0<\delta<1$ we have
\[
| \int_{\R^d} \Phi(\frac{\xi}{U}) \widehat f(\xi) e(U\Phi^{**}(\frac{\xi}{U})) \, \frac{d\xi}{\sqrt{U^d}} |^2   \ll   \delta^{-\frac d2}  \int_{\R^d} |\tilde\Phi(\frac{\xi}V)  \widehat f_{s_{\xi}}(\xi) | \, \frac{d\xi}{\sqrt{ V^d}} +  h_{\delta}
\]
for some function $s$, with $s(\xi)=s_{\xi}$ bounded by $V/U\delta$, with $f_s(x)=f(x+s)\overline f(x)$, $\tilde\Phi\in C_c^{\infty}(\R^d)$  and {\color{black}$h_{\delta}\ll V^{d} \delta^{1/|o(1)|}$}. The function $\tilde\Phi$ does not depend on anything, and the constants implicit in the bounds  are independent of $y$ and depend just on bounds for the derivatives and supports of $\Phi$ and $\Phi^*$.
\end{proposition}
\begin{proof}
Let us call $I$ to the integral we want to bound. Then, we have
\[
 I=\int_{\R^d} \varrho(\frac{\zeta}V) I \, d\frac{\zeta}V
\]
for any $\varrho\in C_c^{\infty}$ real, {\color{black} radial} and supported in the unit ball with $\int_{\R^d}\varrho=1$,   with
\[
 I=\frac{1}{\sqrt{U^d}}\int_{\R^d}  \widehat f(\xi+\zeta) E^U(\xi+\zeta) \, d\xi 
\]
for every $\zeta\in\R^d$, with $E^U(\xi)=E(\xi/U)$, $E(\xi)=\Phi(\xi) e(U(y{\color{black}\xi}+\Phi^*(\xi)))$. Interchanging the integrals and applying Cauchy's inequality, taking into account the support of  $\Phi$ and $\varrho$, we have
\[
 |I|^2 \ll  \int_{\R^d} |\int_{\R^d}    \varrho(\frac{\zeta}V) \widehat f(\xi+\zeta) E^U(\xi+\zeta)    \, d\frac{\zeta}V |^2 \, d\xi.
\]
Expanding the square and making a change of variables we have
\[
 |I|^2\ll  \int \varrho*\varrho (\frac{\zeta}V)   [\int_{\R^d} (\widehat f)_{\zeta}(\xi) \overline{ (\overline E^U)_{\zeta}(\xi)} \, d\xi  ]    \, d\frac{\zeta}V
\]
with $g_{\zeta}(\xi)=g(\xi+\zeta)\overline g(\xi)$. 
Since $\widehat{(\widehat f)_{\zeta}}(x)=\widehat{\overline f_x}(\zeta)$ and $\widehat{(\overline E^U)_{\zeta}}(x)=U^d \widehat{\overline E_{\zeta/U}}(Ux)$, by Plancherel and the change of variable $x\mapsto x/U$ we have
\[
 \int_{\R^d} (\widehat f)_{\zeta}(\xi)\overline{ (\overline E^U)_{\zeta}(\xi)} \, d\xi  = \int_{\R^d} \widehat{\overline f_{x/U}}(\zeta) \overline{\widehat{\overline E_{\zeta/U}}(x)} \, dx.
\]
We have that $E_{\zeta/U}$ has compact support {\color{black}and any partial derivative of order $j\ge 0$} is bounded by $O((1+|\zeta|)^j)$, hence $\widehat{\overline{E}_{\zeta/U}}(x)\ll_j ({\color{black}|x|}/(1+|\zeta|))^{-j}$.
{\color{black} By using it for any $j>d$ we get that
\[
 \int_{|x|\ge V/\delta} |\widehat{\overline{E}_{\zeta/U}}(x)| \, dx \ll_j (1+|\zeta|)^j  (V/\delta)^{d-j}\ll_j V^j (V/\delta)^{d-j} \ll_j V^d \delta^{j-d},
\]
since $1+|\zeta|\ll V$. On the other hand, due to the support of $f$ we have $\|\widehat{f_s}\|_{L^{\infty}}\le \|f_s\|_{L^{1}}\ll_d \|f\|_{L^{\infty}}^2=1 $, and then } 
\[
 \int_{\R^d} \widehat{\overline f_{x/U}}(\zeta) \overline{\widehat{\overline E_{\zeta/U}}(x)} \, dx \ll |\widehat{\overline f_{{\color{black}s_{-\zeta}}}}(\zeta)| \int_{|x|<V/\delta}  |\widehat{\overline E_{\zeta/U}}(x)| \, dx  + h_{\delta}
\]
for some $|{\color{black}s_{-\zeta}}|\le V/U\delta$. {\color{black} Finally, by applying Cauchy's inequality followed by Plancherel we have
\[
 \int_{|x|<V/\delta}  |\widehat{\overline E_{\zeta/U}}(x)| \, dx \ll (V/\delta)^{d/2}  \|E_{\zeta/U}\|_{L^2} \ll (V/\delta)^{d/2},
\]
since $E_{\zeta/U}$ is bounded and compactly supported. Thus, the result follows with $\tilde{\Phi}=\varrho *\varrho$. }
\end{proof}

We are finally ready to control $\lambda_{S,x_0,\beta}^m (f)$ for $f$ a mixing function and $x_0$ a curved point.

\begin{proposition}[Midrange frequencies at curved points]
\label{main_midrange}
Let $S$ be totally curved and $f$ a $(T,\alpha)$-mixing function. Let $0<\beta<1$. There exists $c>1$ depending just on $S$ such that if $x_0$ is at distance $\beta$ from a $c\beta$-curved point of $S$ then we have
\[
\lambda_{S,x_0,\beta}^m(f) \ll  T^{n-\frac{\alpha}2 d} \beta^{-4d}\rho^{-3d/2} [1+R]
\]
with $R\ll T^{\alpha d}\min(\beta^{-1},\beta^3 \rho T)^{-1/|o(1)|}$ {\color{black} and $\rho$ as in (\ref{frequency_splitting})}.
\end{proposition}
\begin{proof}
Apply Proposition \ref{midrange_bound_oscillatory} followed by Proposition \ref{van_der_corput}. Since $f$ is $(T,\alpha)$-mixing the resulting integral can be bounded for any $V\le \delta \rho T^{\alpha}$,  by (\ref{character_condition}) (in order to apply (\ref{character_condition}) to the integral coming from the bound in Proposition \ref{van_der_corput} we only need to use the support of $\tilde\Phi$). Thus, picking $\delta=\beta$ and $V=\beta\rho T^{\alpha}$, since $\rho T<U<T/\rho$  we get that
\[
 \lambda_{S,x_0,\beta}^m(f)\ll  T^n \beta^{-3d} \rho^{-n}\log(1/\rho)  [ \beta^{-d}\rho^{-d/2} T^{-\frac{\alpha}2 d}  +  R] 
\]
with $R\ll T^{\alpha d}\min(\beta^{-1},\beta^3 \rho T)^{-1/|o(1)|}$.
\end{proof}

\begin{proposition}[Mixing case at curved points]
\label{main_local}
Let $S$ be totally curved and $f$ a $(T,\alpha)$-mixing function with $0<\alpha<1$. {\color{black} There exists $c>1$ depending just on $S$ such that} if $x_0$ is a point at distance $O(\beta)$ from a $c\beta$-curved point of $S$ we have
\[
 \lambda_{S,x_0,\beta}(f)\ll \beta^{1/8}  
\]
for any $\beta$ in the range $T^{-\frac{\alpha^2}{20}+\frac{1}{10}\frac nd} < \beta {\color{black}\le} T^{-8\frac nd}$.
\end{proposition}
\begin{proof}
From Propositions \ref{main_low_high} and \ref{main_midrange}, by picking $\rho=\beta^{3/2\alpha}$, we get that
\[
 \lambda_{S,x_0,\beta}(f)\ll  \beta^{3/2\alpha} +\beta^{d/4} T^{n}  + T^{n-\frac{\alpha}2 d} \beta^{-4d}\beta^{-9d/4\alpha} [1+R]
\]
with $R\ll T^{\alpha d} \min(\beta^{-1},\beta^{3+3/2\alpha} T)^{-1/|o(1)|}$, for every $x_0$ at distance $O(\beta)$ from a $c\beta$-curved point of $S$, with $\beta>T^{-\alpha^2/3}$. The result follows in the range chosen for $\beta$.
\end{proof}

Theorem \ref{main_real} finally follows from Propositions \ref{main_local} and \ref{local_to_global}, together with Proposition \ref{main_singular}. {\color{black} More precisely, in the singular case Proposition \ref{main_singular} implies Theorem \ref{main_real}, since $\alpha^2 d/200\le d/200\le d/4$; in the mixing case, the bound $n\le \alpha^2 d/200$ from the statement of Theorem \ref{main_real} allows us to apply Proposition \ref{main_local} with $\beta=T^{-8\frac nd}$, so that
\[
 \lambda_{S,x_0,\beta}(f)\ll \beta^{1/8}=T^{-\frac nd}
\]
for any $x_0$ at distance $O(\beta)$ for a $c\beta$-curved point of $S$. Then, by applying Proposition \ref{local_to_global} with $\beta=T^{-8n/d}$, $\epsilon=\beta^{1/8}$ and $\gamma=c\beta$ we deduce that $\lambda_S(f)\ll \beta^{1/8}+\beta^{1/c_S}$, so Theorem \ref{main_real} follows. 
}

\section*{Acknowledgements}

I worked on this paper as a result of an invitation by E. Lindenstrauss to stay for four months at The Einstein Institute of Mathematics, at the Hebrew University of Jerusalem. I would like to thank him for his help with the problem as well as the people at the Institute for the great conditions to work there.

I would also like to thank F. Chamizo for his comments on the paper, which improved it substantially, and to the referee for careful reading of the manuscript and  for pointing out to me the possibility of proving Theorem \ref{main} for general quotients of $\text{SL}_2(\mathbb R)^d$.

\bibliographystyle{alpha}
\bibliography{submanifolds}

\end{document}